\documentclass[12pt,usenames,dvipsnames]{amsart}

\usepackage{booktabs}

\usepackage{wrapfig}
\usepackage[frozencache,cachedir=minted-cache]{minted}
\usepackage{longtable}
\usepackage{mathpazo}
\usepackage{hyperref}
\usepackage{cleveref}
\usepackage{bm}
\usepackage{comment}
\usepackage[margin=1in]{geometry}
\usepackage{enumitem}
\usepackage{calrsfs}
\DeclareMathAlphabet{\pazocal}{OMS}{zplm}{m}{n}

\usepackage{algorithm2e}

\renewcommand{\KwData}{\textbf{Input:}}
\renewcommand{\KwResult}{\textbf{Output:}}
\usepackage{makecell}
\usepackage{graphicx}
\usepackage{tikz-cd}
\usepackage{multirow}

\usepackage[cmtip, all]{xy}
\usepackage{array}
\usepackage{lscape}

\usepackage{amssymb,amsmath,latexsym,amsthm}
\newcommand{\T}{\mathbb{T}}

\DeclareMathOperator{\val}{val}

\DeclareMathOperator{\an}{an}

\DeclareMathOperator{\trop}{trop}

\DeclareMathOperator{\init}{in}

\DeclareMathOperator{\Trop}{Trop}

\newcommand{\field}[1]{\mathbb{#1}}

\newcommand{\R}{\field{R}}

\usepackage{booktabs}
\usepackage[margin=1in]{geometry}
\usepackage{array}

\newtheorem{theorem}{Theorem}[section]
\newtheorem{lemma}[theorem]{Lemma}
\newtheorem{proposition}[theorem]{Proposition}

\newtheorem*{theorem*}{Theorem}

\theoremstyle{definition}

\newtheorem{example}[theorem]{Example}
\newtheorem{remark}[theorem]{Remark}
\newtheorem*{assumption*}{Assumption}

\theoremstyle{remark}

\newtheorem{definition}[theorem]{Definition}  

\numberwithin{equation}{section}
\numberwithin{table}{section}
\numberwithin{figure}{section}

\title{Faithful tropicalization of hyperelliptic curves}
\author[H. Markwig]{Hannah Markwig}
\address{Hannah Markwig, Universit\"at T\"ubingen \\ Auf der Morgenstelle 10 \\ 72076 T\"{u}bingen \\ Germany}
\email{\href{mailto:hannah@math.uni-tuebingen.de}{hannah@math.uni-tuebingen.de}}

\author[L. Ristau]{Lukas Ristau}
\address{Lukas Ristau, Fraunhofer Institute for Industrial Mathematics\\Fraunhofer-Platz 1\\67663 Kaiserslautern\\Germany
}
\email{\href{mailto:lukas.ristau@itwm.fraunhofer.de}{lukas.ristau@itwm.fraunhofer.de }}
\author[V. Schleis]{Victoria Schleis}
\address{Victoria Schleis, Universit\"at T\"ubingen \\ Auf der Morgenstelle 10 \\ 72076 T\"{u}bingen \\ Germany}
\email{\href{mailto:Victoria.Schleis@student.uni-tuebingen.de}{Victoria.Schleis@student.uni-tuebingen.de}}

\begin{document}
%
\maketitle

\begin{abstract}
We provide explicit faithful re-embeddings for all hyperelliptic curves of genus at most three and an algorithmic way to construct them. Both in the faithful tropicalization algorithm and the proofs of correctness, we showcase \textsc{Oscar}-methods for commutative algebra, polyhedral and tropical geometry. This article is submitted as a bookchapter in the upcoming \textsc{Oscar} book. \end{abstract}

\section{Introduction}

Tropical geometry allows a fruitful exchange of methods between algebraic and convex geometry, as tropicalization preserves important properties of algebraic varieties such as dimension.
The question which properties can be preserved under tropicalization --- and how --- is of foundational importance in order to make use of the exchange of methods.
In this chapter, we concretely answer this question for the case of hyperelliptic curves of genus $3$.

We assume that the reader is familiar with varieties defined over the Puiseux series e.g., their tropicalizations, in particular tropical hypersurfaces and duality, especially tropical plane curves and their dual Newton subdivisions, as in  \cite{MS2015}.

Consider a plane curve and its tropicalization. As dimension is preserved, the latter is a $1$-dimensional polyhedral complex, more precisely a balanced graph in $\mathbb{R}^2$ that is dual to its Newton subdivision, as described in  \cite{MS2015}. Is the genus of the curve reflected in its tropicalization?

Examples of a tropical line and a tropical conic, as in \cite{MS2015}, suggest that the tropicalization of a rational (i.e.\ genus $0$) curve should be a rational graph, i.e.\ a tree.
\begin{example}\label{ex-tropicalweierstrass}

What about the elliptic curve defined by the Weierstra\ss{} equation $$y^2=x^3+4\cdot x^2+8t^4\cdot x,$$ which is of genus $1$? We expect the tropicalization to be a graph containing a cycle, thus reflecting the genus.
 To compute its tropicalization, we use the following \textsc{Oscar} \cite{Oscar}  commands:

  \begin{minted}[breaklines]{jlcon}
julia> Kt,t = RationalFunctionField(QQ,"t")
(Rational function field over QQ, t)

julia> Kxy, (x,y) = Kt["x", "y"]

julia> f=-x^3-4*x^2+y^2+(-8*t^4)*x;

julia> val_t = TropicalSemiringMap(Kt,t)
The t-adic valuation on Rational function field over QQ

julia> ftrop=tropical_polynomial(f,val_t)
x^3 + x^2 + (4)*x + y^2

julia> Tf=TropicalHypersurface(ftrop)
min tropical hypersurface embedded in 2-dimensional Euclidean space

julia> PC = underlying_polyhedral_complex(Tf)
Polyhedral complex in ambient dimension 2

julia> visualize(PC)

  \end{minted}
\end{example}

 \begin{figure}
\begin{center}
\begin{tabular}{c c}
    \includegraphics[width=8.5cm,trim=300 200 100 100,clip]{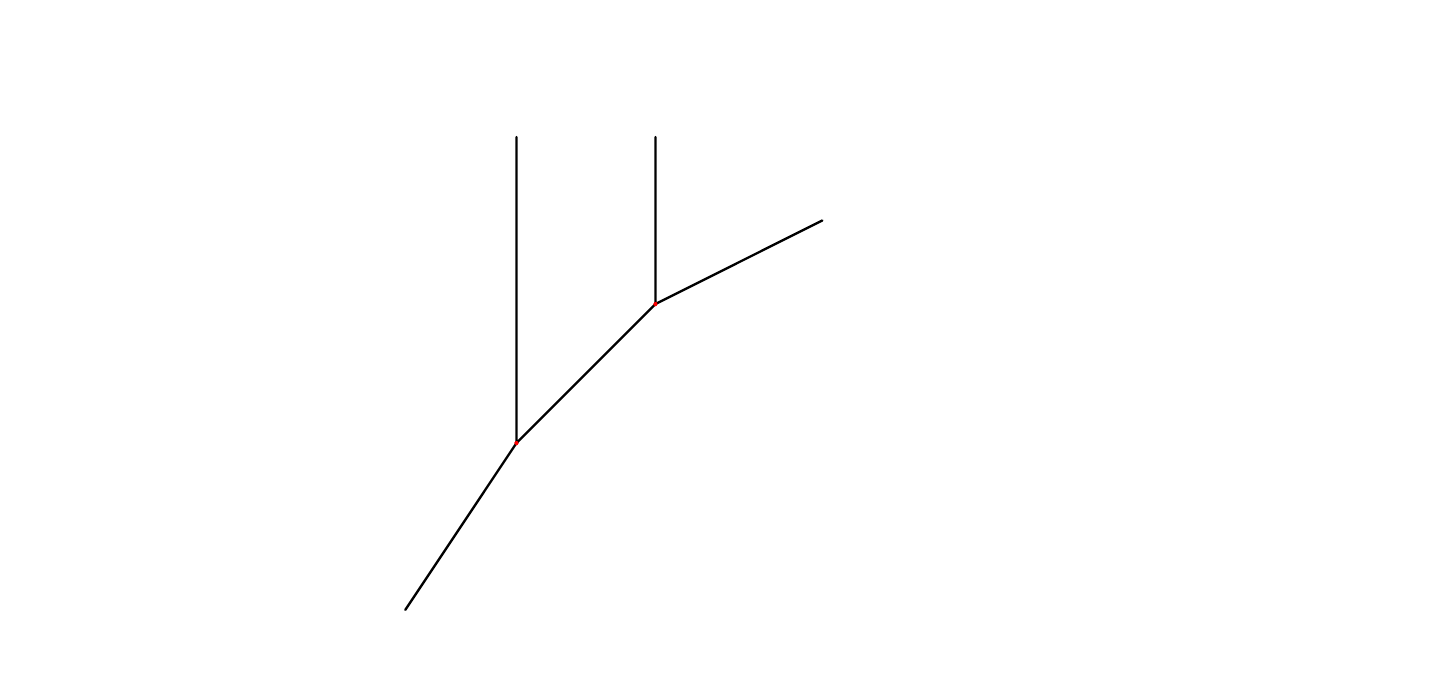}  & 
\hspace{-4cm}\includegraphics[width=8.5cm,trim=300 200 100 100,clip]{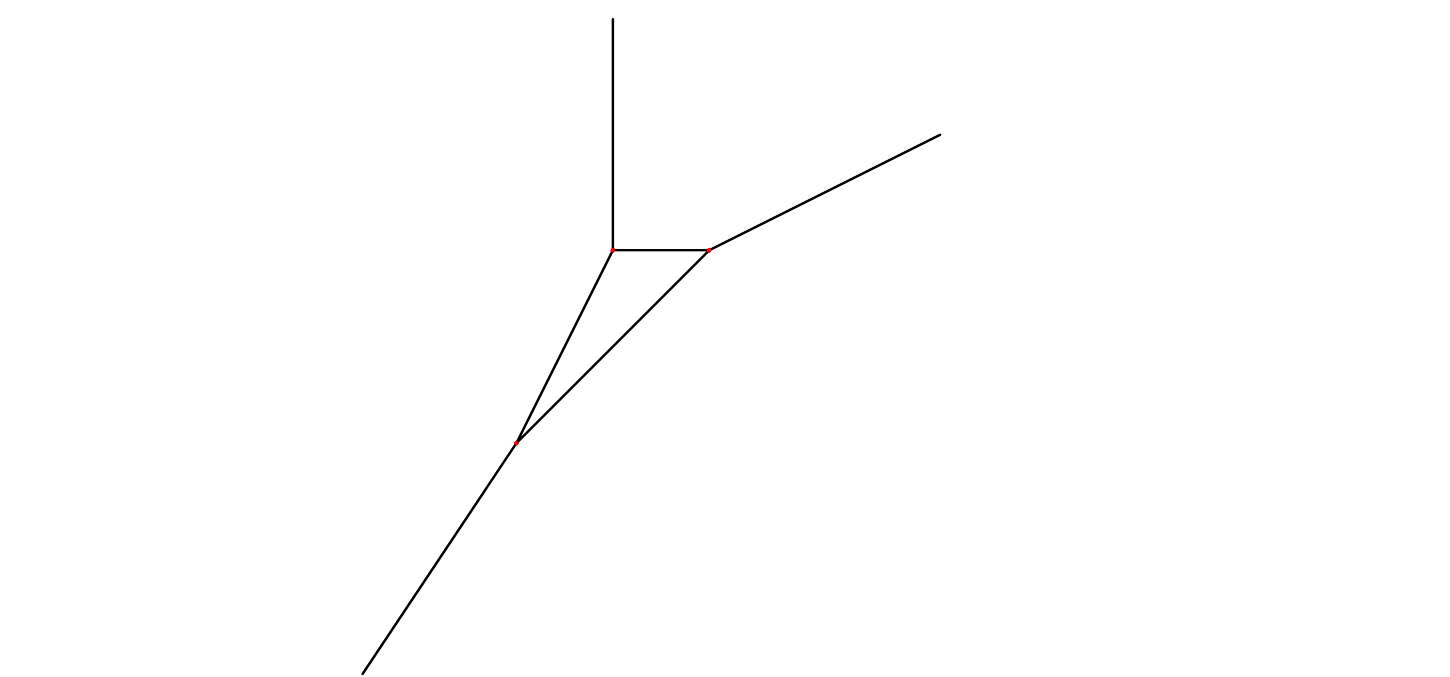}
\end{tabular}
 
\end{center}
   \caption{The tropicalization of the elliptic curve $V(y^2-x^3-4\cdot x^2-8t^4\cdot x)$ is a rational graph, depticted on the left. After the coordinate change, the tropicalization of the elliptic curve from above contains a cycle, depicted on the right.} \label{fig-ellipticrational}
 \end{figure}
 The left of Figure \ref{fig-ellipticrational} shows the output of the computation. It turns out that the tropicalization is again a rational graph, i.e.\ the genus was somehow lost in the process of tropicalization. We should not let this discourage ourselves, as we can see that this loss was happening merely because of an unfavorable choice of embedding:
 Let us change coordinates and define
 $$g=f(x,y-2x)=-x^3-4\cdot xy+y^2-8t^4\cdot x.$$
 Repeating the computation from above with $g$, we obtain the tropicalization depicted on the right of Figure \ref{fig-ellipticrational}, which contains a cycle as desired.

The genus is only a topological invariant: every elliptic curve is of genus $1$, but their various complex structures may be distinguished using the $j$-invariant. Is the $j$-invariant of an elliptic curve reflected in the tropicalization?

As suggested by Mikhalkin in \cite[Example 3.15]{Mik06}, the tropicalization of an elliptic curve defined over a field with a non-Archimedean valuation should contain a cycle whose length is equal to the valuation of the $j$-invariant.

\begin{example}\label{ex-wrongcyclelength}
Consider the elliptic cubic curve given by the equation
\begin{align*} f=&(-t^2)\cdot x^3+(t^{20})\cdot x^2y+(t^2)\cdot xy^2+(t^{14})\cdot y^3+(-3t^3)\cdot x^2+xy\\ &+(t^3+t^5-t^6)\cdot y^2+(-3t^4)\cdot x+(t+t^2)\cdot y+(2t^2+t^5).\end{align*}

Repeating the computation from Example \ref{ex-tropicalweierstrass} with $f$ we obtain the picture on the left in Figure \ref{fig-wrongcyclelength}.

\begin{figure}
 \begin{center}
 \begin{tabular}{c c}  
  \includegraphics[width=8.5cm,trim=300 0 200 220,clip]{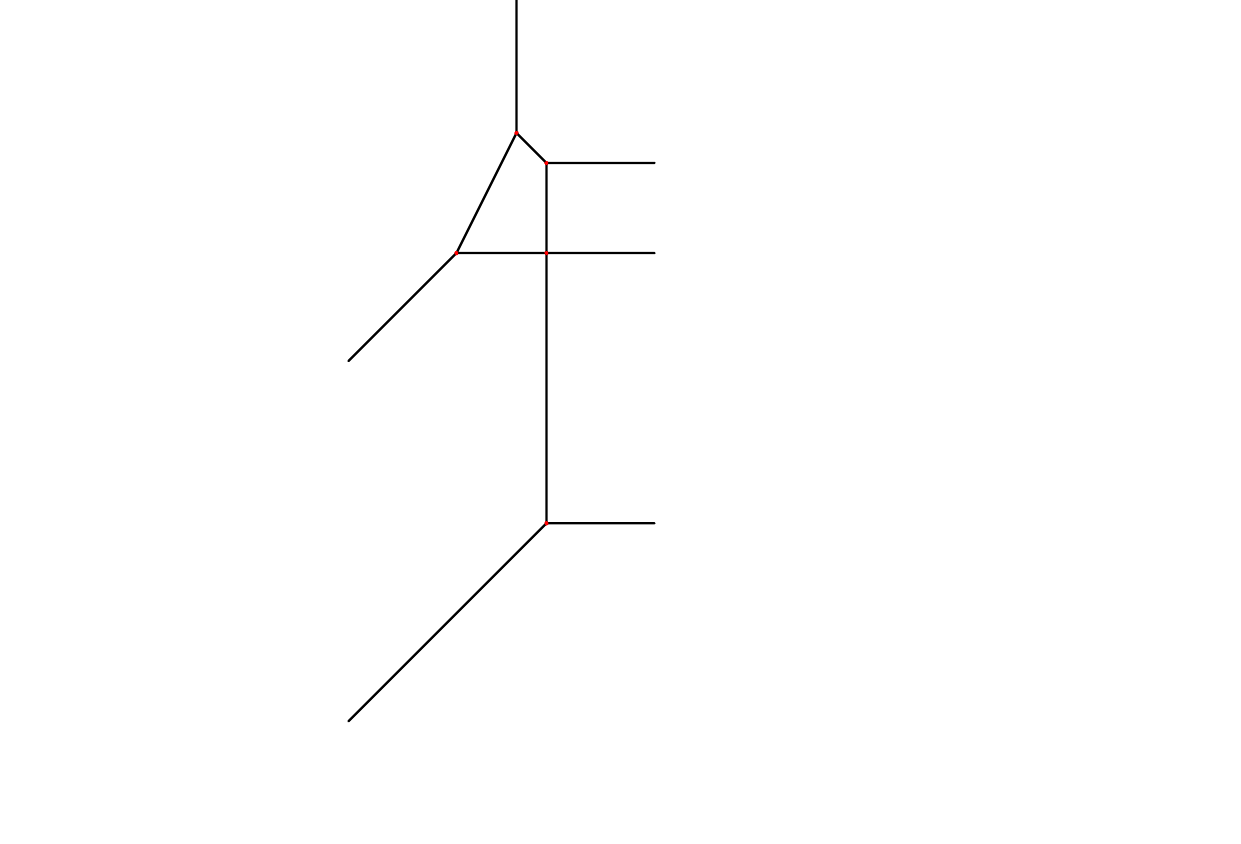}&  
 \hspace{-3cm}\includegraphics[width=8.5cm,trim=300 0 200 220,clip]{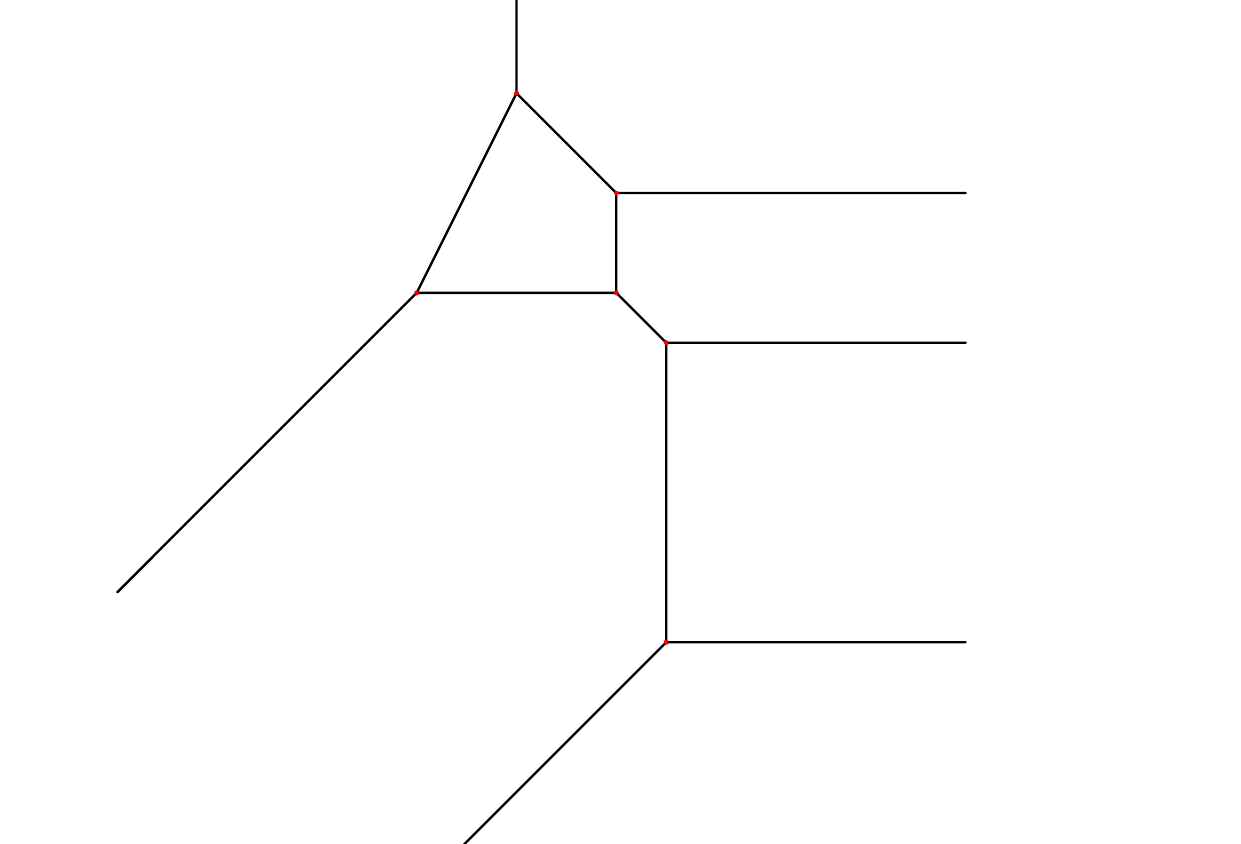}
 \end{tabular}
  \end{center}
  \caption{The tropicalization of the elliptic curve $V(f)$ of Example \ref{ex-wrongcyclelength}, depicted on the left, has a cycle, but its length does not equal the valuation of its j-invariant. After coordinate change, we obtain the expected cycle length, depicted on the right.} \label{fig-wrongcyclelength}
 \end{figure}

After a coordinate change
\begin{align*}
g=f(x-t,y)&=(-t^2)\cdot x^3+(t^{20})\cdot x^2y+(t^2)\cdot xy^2+(t^{14})\cdot y^3+(1-2t^{21})\cdot xy\\&+(t^5-t^6)\cdot y^2+(t^2+t^{22})\cdot y+(2t^2+2t^5)
\end{align*}

we obtain the picture on the right in Figure \ref{fig-wrongcyclelength}.

Using the command \verb!vertices(PC2)! for the underlying polyhedral complex \verb!PC2! of our tropical plane curve
we obtain the coordinates of the vertices:

\begin{minted}{jlcon}
julia> vertices(PC2)
6-element SubObjectIterator{PointVector{QQFieldElem}}:
 [3, -9]
 [-2, -2]
 [2, -2]
 [3, -3]
 [2, 0]
 [0, 2]
\end{minted}

Before the coordinate change, we had the vertices:

\begin{minted}{jlcon}
julia> vertices(PC1)
5-element SubObjectIterator{PointVector{QQFieldElem}}:
 [1, -11]
 [1, -2]
 [-2, -2]
 [0, 2]
 [1, 1]    
\end{minted}

Thus, before the coordinate change we obtained as cycle length 9, while we obtain 10 after the coordinate change.

 We compute the j-invariant in \textsc{Oscar} \cite{Oscar}  using the function
 \verb!j_invariant_cubic(f)! as supplied at \url{https://victoriaschleis.github.io/bookchapter}. It turns out that 10 is the valuation of the j-invariant and thus only the second embedding, after the coordinate change, reflects the j-invariant.

\end{example}

Both examples illustrate that the tropicalization of a curve depends on the chosen embedding.
We can see that for some, but not all, choices of embeddings, many properties of an elliptic curve are reflected in its tropicalization: dimension, genus, and $j$-invariant.
This motivates the following definition:

\begin{definition}[Faithful tropicalizations of elliptic curves]
An embedding of an elliptic curve $E$ with $j$-invariant $j(E)$ is called \emph{faithful} if its tropicalization is a graph containing a cycle of length $\val(j(E))$.\label{def-faithfulell}
\end{definition}
A faithful embedding of an elliptic curve can be viewed as optimal from the point of view of tropical geometry, as all important features are preserved to the zeroth order.

The fact that some, but not all, embeddings are faithful motivates the study of the limit of all tropicalizations, where tropicalization depends on the chosen embedding.

\begin{definition}[Analytification is the limit of all tropicalizations, \cite{P2009}, Theorem 1.1]
Let $\chi$ be an affine variety over $K$. Let $e$ be an affine embedding of $\chi$.  Then we define the \emph{Berkovich space} (or \emph{analytification}) $\chi^{\an} := \varprojlim \Trop(e(\chi))$, i.e. {analytification is the limit of all tropicalizations}.
\end{definition}
Historically, the analytification is defined differently (in terms of extensions of valuations) and the statement above is a theorem. For the purpose of this text however, we can take it as a definition. The statement holds analogously for projective curves after compactification, resp.\ by gluing affine charts.

\begin{example}
The tropicalization of a $\mathbb{P}^1$ can be viewed as the compactification of $\mathbb{R}$. If we embed our line into the (torus of a) plane, e.g.\ as $V(x+y+1)$ and tropicalize, we obtain the well-known picture of a tripod, which  projects onto $\mathbb{R}$. Next, we embed our line into threespace. After tropicalization, several combinatorial types may occur. In general, we obtain a tropical curve with 2 vertices and one bounded edge as in the middle of Figure \ref{fig-line}. Continuing this procedure and taking the limit of all tropicalizations, we obtain as analytification of a line
an infinite tree as sketched in the right of Figure \ref{fig-line} which is often called the witch's broom.

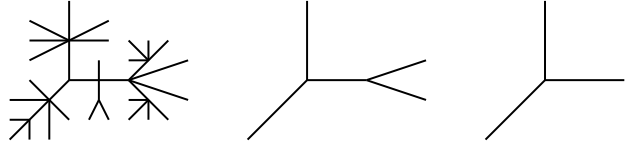
\begin{figure}\centering

\tikzset{every picture/.style={line width=0.75pt}} 

\begin{tikzpicture}[x=0.75pt,y=0.75pt,yscale=-1,xscale=1]

\draw    (-100,-180) -- (-140,-180) ;
\draw    (-140,-180) -- (-170,-150) ;
\draw    (-140,-220) -- (-140,-180) ;
\draw    (-200,-170) -- (-230,-180) ;
\draw    (-200,-190) -- (-230,-180) ;
\draw    (-230,-180) -- (-260,-180) ;
\draw    (-260,-220) -- (-260,-180) ;
\draw    (-260,-180) -- (-290,-150) ;
\draw    (-350,-180) -- (-330,-160) ;
\draw    (-320,-170) -- (-350,-180) ;
\draw    (-320,-190) -- (-350,-180) ;
\draw    (-350,-180) -- (-380,-180) ;
\draw    (-380,-220) -- (-380,-180) ;
\draw    (-380,-180) -- (-410,-150) ;
\draw    (-330,-200) -- (-350,-180) ;
\draw    (-360,-200) -- (-380,-200) ;
\draw    (-360,-210) -- (-380,-200) ;
\draw    (-380,-200) -- (-400,-190) ;
\draw    (-380,-200) -- (-400,-200) ;
\draw    (-380,-200) -- (-400,-210) ;
\draw    (-380,-160) -- (-390,-170) ;
\draw    (-390,-150) -- (-390,-170) ;
\draw    (-390,-170) -- (-400,-180) ;
\draw    (-390,-170) -- (-410,-170) ;
\draw    (-340,-170) -- (-350,-160) ;
\draw    (-340,-170) -- (-350,-170) ;
\draw    (-340,-190) -- (-350,-190) ;
\draw    (-340,-190) -- (-340,-200) ;
\draw    (-340,-190) -- (-350,-200) ;
\draw    (-340,-160) -- (-340,-170) ;
\draw    (-365,-170) -- (-365,-180) ;
\draw    (-365,-180) -- (-365,-190) ;
\draw    (-370,-160) -- (-365,-170) ;
\draw    (-360,-160) -- (-365,-170) ;
\draw    (-400,-150) -- (-400,-160) ;
\draw    (-410,-160) -- (-400,-160) ;

\end{tikzpicture}

    \caption{Two tropicalizations and a sketch of the analytification of a line.} \label{fig-line}
 \end{figure}

\end{example}

Analytifications satisfy many desirable properties. They are
infinitely branched metric graphs, \cite[Section 2.4]{BPR2014}. In some sense, analytifications can be viewed as generalizations of the $j$-invariant for elliptic curves: the analytification of an elliptic curve is an infinitely branched graph containing precisely one cycle of length $\val(j)$. More generally, the analytification of a curve $\chi$ of genus $g$ contains a finite metric subgraph of genus $g$ as a deformation retract, which is called the \emph{(minimal) Berkovich skeleton} of $\chi$. In the case of an elliptic curve, the minimal Berkovich skeleton is a loop of length $\val(j)$.

Using the analytification, we can extend our notion of faithful tropicalization to arbitrary curves:

\begin{definition}[Faithful tropicalization]
Let $\chi$ be a projective curve and $e:\chi\rightarrow \mathbb{P}^n$ an embedding of $\chi$ into projective space. The embedding $e$ of $\chi$ is called \emph{faithful} on a subgraph $\Gamma$ of the analytification $\chi^{\an}$ if $\Trop(e(\chi))$ contains an isometric copy of $\Gamma$. When not otherwise specified, we consider faithfulness with respect to the minimal Berkovich skeleton.
\end{definition}
If we fix  the cycle as the subgraph $\Gamma$ of the analytification of an elliptic curve, we recover Definition \ref{def-faithfulell}.

Faithful tropicalizations are the best tropicalizations we can hope for, in the sense that they reflect as much of the geometry of the algebraic curve as possible. Analytifications also reflect the geometry, but they are infinite objects which are hard to control. A tropicalization which is chosen arbitrarily is accessible for computations, but may not reflect the features we want to observe. A faithful tropicalization combines the advantages of both: it is a finite object which is easier to control, but it still reflects the geometric features we want to observe.

In \cite[Theorem 5.20]{BPR2014}, it is shown that faithful tropicalizations exist for every curve. But how can we find them concretely?

In this paper, we answer this question for hyperelliptic curves of genus $3$. The technique we use is called re-embedding into tropical modifications and described in detail in Section \ref{sec_reembeddings_and_blocks}.
\begin{theorem}\label{thm-main}
Let $\chi$ be a hyperelliptic curve of genus $3$. Then Table \ref{table_reembedding} gives re-embedding polynomials such that the re-embedding into the tropical modification induced by the polynomials is a faithful tropicalization of $\chi$.
\end{theorem}


Faithful tropicalizations of curves of genus $2$ have been studied in \cite{CM2019}. The concrete construction of a faithful tropicalization depends on a case-by-case analysis, considering different possibilities for equalities and inequalities among the valuations of the coefficients.
As the Newton polygon of such a curve is sufficiently small, it was possible to perform this case-by-case analysis by hand for curves of genus $2$.

Generalizing to hyperelliptic curves of genus $3$ enlarges the amount of possibilities for inequalities among the valuations of the coefficients and the Newton subdivisions to such an extend that a case-by-case analysis by hand is out of reach. We have obtained Theorem \ref{thm-main} with the aid of \textsc{Oscar} \cite{Oscar} .

\begin{center}
    \textsc{Acknowledgements}
\end{center}

 We thank Yue Ren for his feedback on a previous version. The authors were supported by the Deutsche Forschungsgemeinschaft (DFG, German Research
	Foundation), Project-ID 286237555, TRR 195.

\section{Techniques for faithful re-embeddings}\label{sec_reembeddings_and_blocks}

In the following, we analyze hyperelliptic curves over a valued field $K$ of characteristic 0 given by a defining equation
\begin{equation}\label{eq_defining}
  g: \hspace{0.3cm}  y^2 = x\cdot\prod_{i=1}^{8}(x_i-\alpha_i);
\end{equation}
and we denote the valuations of coefficients by $\omega_i = \val(\alpha_i)$.

In the introduction we showed an example of a coordinate change that resulted in a faithful tropicalization. In this section, we will introduce techniques to construct these coordinate changes methodically, by using \emph{re-embeddings} into \emph{tropical modifications}.

The faithful re-embedding construction we describe follows the divide-and-conquer principle.
 Given a hyperelliptic curve $\chi$ of genus $g\leq 3$ defined by an equation $g$, see \eqref{eq_defining}, we \emph{separate} the tropical hyperelliptic curve into subgraphs according to conditions on the initial forms of $g$ using Table \ref{table_standard_trop_coeff_cond}. For each subgraph identified this way, we change coordinates by \emph{finding a correct re-embedding polynomial $f_i$} using Table \ref{table_reembedding}.
For all subgraphs connected by bridges, we further \emph{reduce the dimension} of the ambient space of the re-embedding via Theorem \ref{thm_mainthm_bridges}.

In this section we show how the Berkovich skeleton can be decomposed into subgraphs that can be considered independently when constructing the re-embedding. Further, we provide details on the re-embeddings themselves. Finally, we describe how to to check faithfulness on a tropical curve.

Throughout this section, we will follow the example of the elliptic curve $V(y^2 -x^3-4x^2-8t^4x)$ on the left in Figure \ref{fig-ellipticrational} and construct its faithful tropicalization using the methods we introduce.
\subsection{Classifying hyperelliptic curves via their Berkovich skeleton.}

Berkovich skeleta of hyperelliptic curves can be glued from a small collection of possible subgraphs, called \emph{building blocks}. In our construction of faithful re-embeddings, these subgraphs serve a vital role: It is sufficient to determine the faithful re-embeddings of building blocks as re-embeddings can be glued with methods we present in Section \ref{sec_globalizing}.

\begin{remark}\label{description subgraph at xi}
    The types of subgraphs that can occur are listed in Table \ref{table_standard_trop_coeff_cond}. A subgraph at $x^i$ is the skeleton whose leftmost vertex is dual to an edge in the Newton subdivision incident to $x^i$. If the subgraph consists of only a point, $x^i$ is the leftmost point on a length $k$ vertical edge, corresponding to the multiplicity $k$ edge of the tropical curve adjacent to the point.
\end{remark}

\begin{lemma}\label{lem_hyperell_involution}
The minimal Berkovich skeleton of a hyperelliptic curve is built by gluing building blocks of the types given in Table \ref{table_standard_trop_coeff_cond}, and the coefficient conditions in Table \ref{table_standard_trop_coeff_cond} hold.
\end{lemma}
\begin{proof}
The types of subgraphs that can appear in the Berkovich skeleton are a special case of \cite[Theorem 1.1]{corey2021hyperelliptictype}. The coefficient conditions for cycles, 2-thetas, points and bridges are due to (or straightforward generalizations of) the conditions given in \cite[Table 5.1]{CM2019}, and the correctness of the conditions for 3-thetas will be shown later in Theorem \ref{thm_3-thetas_faithful}.
\end{proof}

  \begin{table}[htbp]
\centering 
\small
\begin{tabular}{cccc}
\toprule 

\makecell{\textbf{Subgraph}\\ \textbf{at} $x^i$}& \makecell{\textbf{Local}\\ \textbf{Berkovich} \\ \textbf{skeleton}} & \makecell{ \textbf{ Local picture in the} \\ \textbf{ Newton subdivision}} & \makecell{ \textbf{Coefficient}\\ \textbf{conditions}} \\ \midrule \midrule

\makecell{$3$-Theta }& \begin{minipage}{2.5cm}\centering

\begin{tikzpicture}
            \draw(1,1.5) arc (90:-90:0.5);

            \draw(0,0.5) arc (90:-90:-0.5);
            \draw(0,0.5) -- (1,0.5);
            \draw(0,1.5) -- (1,1.5);
            \draw(0,1.5) -- (0,0.5);
            \draw(1,1.5) -- (1,0.5);
            \fill (0,0.5) circle (1.5pt);
            \fill (0,1.5) circle (1.5pt);
            \fill (1,0.5) circle (1.5pt);
            \fill (1,1.5) circle (1.5pt);
          \end{tikzpicture}

\end{minipage}
& \begin{minipage}{4cm} \centering \begin{tikzpicture}[scale = 0.55]
\draw[opacity = 1]
(0,2) -- (1,0)

(1,0) -- (2,0)
(6,0) -- (7,0)
(0,2) -- (7,0);

7

\draw[purple]
(0,2) -- (2,0)
(0,2) -- (4,0)
(0,2) -- (6,0)
(2,0) -- (6,0)
;

\fill[opacity = 0.5] (0,0) circle (2pt);
\fill (1,0) circle (2pt);
\fill (2,0) circle (2pt);
\fill[purple] (3,0) circle (2pt);
\fill (4,0) circle (2pt);
\fill[purple] (5,0) circle (2pt);
\fill (6,0) circle (2pt);
\fill (7,0) circle (2pt);
\fill[opacity = 0.5] (0,1) circle (2pt);
\fill[purple] (1,1) circle (2pt);
\fill[purple] (2,1) circle (2pt);
\fill[purple] (3,1) circle (2pt);
\fill (0,2) circle (2pt);

\end{tikzpicture} \end{minipage}

& \makecell{
$\omega_1 < \omega_{2}< \omega_{3} < \omega_{5}< \omega_{7}< \omega_8$;
\\  $\omega_3 = \omega_4,\omega_{5} =  \omega_{6}$ \\  $\init(\alpha_{3}) = \init(\alpha_{4}), \init(\alpha_{5}) = \init(\alpha_{6})$ }  \\ \midrule
\makecell{$2$-Theta }& \begin{minipage}{2.5cm}\centering

\begin{tikzpicture}
            \draw(0,1.5) arc (90:-90:0.5);

            \draw(0,0.5) arc (90:-90:-0.5);
            \draw(0,1.5) -- (0,0.5);
            \fill (0,0.5) circle (1.5pt);
            \fill (0,1.5) circle (1.5pt);
          \end{tikzpicture}

\end{minipage}
& \begin{minipage}{4cm} \centering \begin{tikzpicture}[scale = 0.55]
\draw[opacity = 1]
(0,2) -- (1,0)

(1,0) -- (2,0)
(4,0) -- (7,0)
(0,2) -- (7,0);

7

\draw[purple]
(0,2) -- (2,0)
(0,2) -- (4,0)
(2,0) -- (4,0)
;

\fill[opacity = 0.5] (0,0) circle (2pt);
\fill (1,0) circle (2pt);
\fill (2,0) circle (2pt);
\fill[purple] (3,0) circle (2pt);
\fill (4,0) circle (2pt);
\fill (5,0) circle (2pt);
\fill (6,0) circle (2pt);
\fill (7,0) circle (2pt);
\fill[opacity = 0.5] (0,1) circle (2pt);
\fill[purple] (1,1) circle (2pt);
\fill[purple] (2,1) circle (2pt);
\fill[opacity = 0.5] (3,1) circle (2pt);
\fill (0,2) circle (2pt);

\end{tikzpicture} \end{minipage}

& \makecell{
$\omega_i < \omega_{i+1} < \omega_{i+2} < \omega_{i+4} <\omega_{i+5}$\\ $i$ odd;
\\  $\omega_{i+2} = \omega_{i+3};$ \\  $\init(\alpha_{i+2}) = \init(\alpha_{i+3}$) }  \\ \midrule
\makecell{Cycle}& \begin{minipage}{2.5cm}\centering

\begin{tikzpicture}
            \draw(0,1.5) arc (90:-90:0.5);

            \draw(0,0.5) arc (90:-90:-0.5);
          \end{tikzpicture}

\end{minipage}
& \begin{minipage}{4cm} \centering \vspace{0.1cm}\begin{tikzpicture}[scale = 0.55]
\draw[opacity = 1]
(0,2) -- (1,0)

(1,0) -- (2,0)
(2,0) -- (7,0)
(0,2) -- (7,0);

7

\draw[purple]
(0,2) -- (2,0)
;

\fill[opacity = 0.5] (0,0) circle (2pt);
\fill (1,0) circle (2pt);
\fill (2,0) circle (2pt);
\fill (3,0) circle (2pt);
\fill (4,0) circle (2pt);
\fill (5,0) circle (2pt);
\fill (6,0) circle (2pt);
\fill (7,0) circle (2pt);
\fill[opacity = 0.5] (0,1) circle (2pt);
\fill[purple] (1,1) circle (2pt);
\fill[opacity = 0.5] (2,1) circle (2pt);
\fill[opacity = 0.5] (3,1) circle (2pt);
\fill (0,2) circle (2pt);

\end{tikzpicture}\vspace{0.1cm}\end{minipage}

& \makecell{
$\omega_i < \omega_{i+1} < \omega_{i+2}$\\ $i$ odd }
\\ \midrule
\makecell{ $k$-point} &
\begin{minipage}{2.5cm}\centering
\begin{tikzpicture}[scale=0.5]

             \fill[opacity = 1] (0, 0) circle (3.5pt);
\node [below] at (0,-0.2) {$k$};
          \end{tikzpicture}
     \end{minipage}
           & \begin{minipage}{4cm} \centering \begin{tikzpicture}[scale = 0.55]
\draw[opacity = 1]
(0,2) -- (1,0)
(0,2) -- (7,0);


\draw[purple]
(1,0) -- (7,0)
;

\fill[opacity = 0.5] (0,0) circle (2pt);
\fill (1,0) circle (2pt);
\fill[purple] (2,0) circle (2pt);
\fill[purple] (3,0) circle (2pt);
\fill [purple](4,0) circle (2pt);
\fill[purple] (5,0) circle (2pt);
\fill[purple] (6,0) circle (2pt);
\fill (7,0) circle (2pt);

\fill[opacity = 0.5] (0,1) circle (2pt);
\fill[opacity = 0.5] (1,1) circle (2pt);
\fill[opacity = 0.5] (2,1) circle (2pt);
\fill[opacity = 0.5] (3,1) circle (2pt);
\fill[opacity = 0.5] (4,1) circle (2pt);

\fill (0,2) circle (2pt);

\end{tikzpicture} \end{minipage}& \makecell{$\omega_i < \omega_{i+1} < \omega_{i+2k+1}$ \\  $\omega_{i+1} = \dots = \omega_{i+2k}$ \\ $\init(\alpha_l) \neq \init(\alpha_j)$ \\ for $i < l < j < 2k+2$}\\ \midrule
  \midrule Bridge &
\begin{minipage}{2.5cm}\centering
          \begin{tikzpicture}[scale=1]
            \draw(-1,1.5) arc (90:-90:0.5);

            \draw(1,0.5) arc (90:-90:-0.5);

			\draw[purple]
			(-0.5,1) -- (0.5,1);

            \fill[opacity = 1] (0.5, 1) circle (1.5pt);
            \fill[opacity = 1] (-0.5, 1) circle (1.5pt);

          \end{tikzpicture}
          \end{minipage}
 &\begin{minipage}{4cm} \vspace{0.1cm} \centering
 \begin{tikzpicture}[scale = 0.55]
\draw[opacity = 1]
(0,2) -- (1,0)

(1,0) -- (2,0)
(0,2) -- (7,0);

\draw[purple]
(0,2) -- (5,0);

\draw
(2,0) -- (7,0)
;

\fill[opacity = 0.5] (0,0) circle (2pt);
\fill (1,0) circle (2pt);
\fill (2,0) circle (2pt);
\fill (3,0) circle (2pt);
\fill (4,0) circle (2pt);
\fill (5,0) circle (2pt);
\fill (6,0) circle (2pt);
\fill (7,0) circle (2pt);

\fill[opacity = 0.5] (0,1) circle (2pt);
\fill[opacity = 0.5] (1,1) circle (2pt);
\fill[opacity = 0.5] (2,1) circle (2pt);
\fill[opacity = 0.5] (3,1) circle (2pt);
\fill[opacity = 0.5] (4,1) circle (2pt);

\fill (0,2) circle (2pt);

\end{tikzpicture}\vspace{0.1cm}
\end{minipage} &
\makecell{$\omega_{i-1} < \omega_i < \omega_{i+1}$ \\ $i$ odd}

\\ \midrule
\makecell{ Point \\ connector}&
\begin{minipage}{2.5cm}\centering
\begin{tikzpicture}[scale=1]
            \draw(-1,1.5) arc (90:-90:0.5);

            \draw(0,0.5) arc (90:-90:-0.5);

            \fill[purple, opacity = 1] (-0.5, 1) circle (2pt);

          \end{tikzpicture}
\end{minipage} &\begin{minipage}{4cm} \centering  \begin{tikzpicture}[scale = 0.55]
\draw[opacity = 1]
(0,2) -- (1,0)

(1,0) -- (2,0)
(6,0) -- (7,0)
(0,2) -- (7,0)
(4,0) -- (6,0);

7

\draw[purple]
(0,2) -- (2,0)
(0,2) -- (4,0)
(2,0) -- (4,0)
;

\fill[opacity = 0.5] (0,0) circle (2pt);
\fill (1,0) circle (2pt);
\fill (2,0) circle (2pt);
\fill[purple] (3,0) circle (2pt);
\fill (4,0) circle (2pt);
\fill(5,0) circle (2pt);
\fill (6,0) circle (2pt);
\fill (7,0) circle (2pt);
\fill[opacity = 0.5] (0,1) circle (2pt);
\fill[purple] (1,1) circle (2pt);
\fill[purple] (2,1) circle (2pt);
\fill[opacity = 0.5]  (3,1) circle (2pt);
\fill (0,2) circle (2pt);

\end{tikzpicture}
\end{minipage}
& \makecell{
$\omega_i < \omega_{2j}$ \\ $i$ odd, $1\leq j \leq k+1$;
\\  $\omega_{i+2l} = \omega_{i+2l+1};$ \\   $\init(\alpha_{i+2l}) \neq \init(\alpha_{i+2l+1});$ \\ $ 1 \leq l \leq k-1;$}  \\ \bottomrule
\end{tabular}
\caption{The different building blocks appearing for hyperelliptic curves and their coefficient conditions. While point connectors and thetas have the same local Newton subdivision, their coefficient conditions and thus their behaviour under modification are distinguishable.  See Remark \ref{description subgraph at xi} for details on notation.\label{table_standard_trop_coeff_cond}}
\normalsize
\end{table}

\subsection{Re-embeddings.}
 By re-embedding a plane curve, we can lift its tropicalization in $\mathbb{R}^2$ to a tropical modification of $\mathbb{R}^2$, hoping to make new features visible in the newly attached cells and thus getting closer to the Berkovich skeleton (see, for instance, Figure \ref{modifying_along_tropical_line} and Example \ref{ex_1-thetas_reduced-dim}).  Formally, they are defined as follows:
\begin{definition}[Tropical modifications]\label{def_trop_polynomial_semiring}

Let $F$ be a polynomial over the tropical semiring $\T$. To each cell $\sigma$ of the graph of $F$ that does not have codimension 0 we attach a new cell $\varsigma$ spanned by $\sigma$ and $e_{n+1} := (0, \dots , 0, 1)$. We obtain a pure rational polyhedral complex in $\R^{n+1}$, the \emph{modification} of $\R^n$ along $F$.
\end{definition}

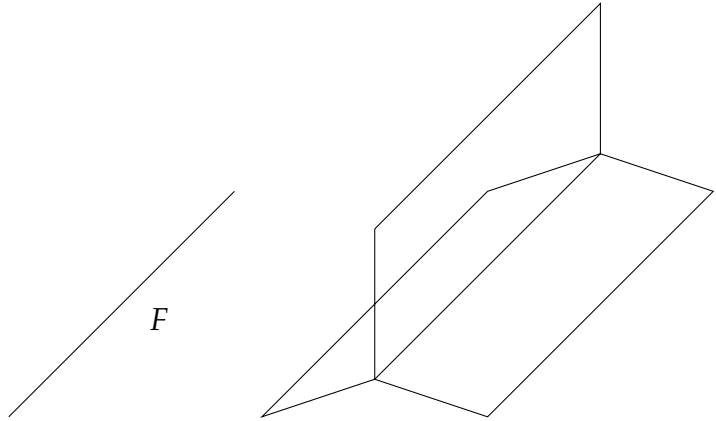
\begin{figure}
    \centering
\begin{tabular}{c c}
\begin{tikzpicture}
\draw (0,0) -- (3,3);
\node[above] at (2,1){$F$};
\end{tikzpicture} &
\begin{tikzpicture}
\draw
(0,2) -- (0,0) -- (3,3) -- (3,5)-- (0,2)
(0,0) -- (-1.5,-0.5) -- (1.5,2.5) -- (3,3)
(0,0) -- (1.5,-0.5) -- (4.5,2.5) -- (3,3)
;

\end{tikzpicture}  \\
\end{tabular}
    \caption{The modification of $\mathbb{R}^2$ along $F = \min\{x+2,y\}$.}
    \label{modifying_along_tropical_line}
\end{figure}

We can write the modification of $F$ as a tropicalization by finding appropriate \emph{polynomial lifts} $f$ of $F$, i.e. polynomials $f$ whose tropicalization is $F$:
$
f(x) = \sum_{\beta \in \text{supp(F)}} c_{\beta}x^{\beta}.
$

\begin{definition}\label{notation}
In the remainder of the text, we will use the following notation:
\begin{itemize}
    \item $\chi$ denotes a hyperelliptic curve and $g$ its defining equation of the form given in  \eqref{eq_defining};
    \item uppercase letters will be used to denote objects in the tropical world, in particular, we will use $F$ (or $F_i$) to denote the tropical polynomials in the tropical polynomial ring $\T[X,Y]$ along which we modify $\R^2$, and call them \emph{modification polynomials};
    \item lowercase letters will be used to denote objects over the valued field $K$. In particular, $f$ (or $f_i$) will denote a polynomial lift of a modification polynomial. We will call these polynomials \emph{re-embedding polynomials}.
\end{itemize}
\end{definition}

\begin{definition}[\cite{CM2019}]\label{tropical_modification_lemma}
Let $F_1, \dots F_n \in \T[X,Y]$ and $f_1 \dots, f_n$ be lifts of the respective $F_i$. Let $g$ be a defining equation for $\chi$ (see \eqref{eq_defining}). Then the tropicalization of the variety $V(I_{g,f})$ of the ideal $
I_{g,f} = \langle g, z_1-f_1, \dots, z_n -f_n \rangle \subset K[x^{\pm}, y^{\pm}, z_1^{\pm}, \dots,  z_n^{\pm}]
$
is a tropical curve in the modification of $\R^2$ along the $F_i$.
\end{definition}
\begin{definition}
     In the following, we will denote by $I_{g,f}$ the ideal corresponding to the re-embedding of the hyperelliptic curve $\chi$ given by a defining equation $g$ (see \eqref{eq_defining}) into the tropical modification induced by the re-embedding polynomials $f = (f_1, \dots, f_n)$.
\end{definition}

\subsection{Analyzing re-embeddings.}

Re-embeddings into tropical modifications only make hidden parts of the tropical curve visible if the modification curve passes through the parts we want to investigate. Then, we can uncover the hidden parts of the curve by investigating what happens on the new cells.
As hidden cycles occur on edges with multiplicity greater than one where two edges get mapped onto each other by the embedding, we restrict our attention to re-embedding polynomials of the form $f(x,y) = y - a_1x - a_2x^2 - \dots - a_nx^n \in K[x,y]$. Their tropicalizations pass through the multiplicity two edges of the tropical hyperelliptic curve before re-embedding, see Table \ref{table_standard_trop_coeff_cond}. An example of this can be found in Figure \ref{fig:running-inside-modification}.

After re-embedding the tropical curve into the modification of $\mathbb{R}^2$, we can check for faithfulness. Since re-embedding is compatible with projection, i.e.\ projections of curves inside a tropical modification are tropical curves under projection \cite{CM2019, S2020}, re-embeddings can be analyzed by investigating projections.

\begin{lemma}[\cite{CM2019}, Lemma 3.3; and \cite{S2020}, Lemma 1.1]\label{lem_proj_lemma_genus_n}
Given an irreducible curve $\chi \subset (K^*)^2$ defined by polynomials $g \in K[x,y]$ and  $f(x,y) = y - a_1x - a_2x^2 - \dots a_nx^n \in K[x,y]$, the tropicalization induced by the ideal $I_{g,f} = \langle g, z-f \rangle \in K[x^{\pm}, y^{\pm}, z^{\pm}]$ is completely determined by the tropical plane curves $\trop(V(g))$, $\trop(V(I_{g,f}\cap K[x^{\pm}, z^{\pm}]))$, $\trop(V(I_{g,f}\cap K[y^{\pm}, z^{\pm}]))$ and one further projection onto a rational plane $P$ in $\R^3$ that is not parallel to any top-dimensional cell in the tropicalization of $z-f$ and not parallel to any of the coordinate planes.
\end{lemma}
    \begin{figure}
        \centering
    \begin{tikzpicture}
    \draw
(0,2) -- (0,0) -- (3,3) -- (3,5)-- (0,2)
(0,0) -- (-1.5,-0.5) -- (1.5,2.5) -- (3,3)
(0,0) -- (1.5,-0.5) -- (4.5,2.5) -- (3,3)
;
\draw[ very thick]
(0.25,-0.1) -- (0.75,0.75) -- (2.25,2.25) -- (3.75,2.75)
(2.25,2.25) -- (2.25, 2.75)
(0.75,0.75) -- (0.75,1.75)
;
\draw[blue, very thick]
(0.75,0.75) -- (1.75,2.25)  -- (2.25,2.25)
(1.75,2.25) -- (1.75, 3.75)
;
\draw
(5.5, 2.5) -- (6,3.5) -- (7,4.5) -- (8,5)
(6,3.5) -- (6,5)
(7,4.5) -- (7,5)
;
\draw[blue,dashed]
(6,3.5) -- (7,4.5);
\draw
(5.5,-0.5) -- (6,0.5) -- (7,1.5) -- (8,2)
 ;
 \draw[blue]
(6,0.5) -- (6.33,1.5) -- (7,1.5) -- (6.33,1.5) -- (6.33,2)
 ;
 \draw[opacity = 0.5]
(5.4,-0.6) -- (8.1, -0.6) -- (8.1, 2.1) -- (5.4,2.1) -- cycle
(5.4,2.4) -- (8.1, 2.4) -- (8.1, 5.1) -- (5.4,5.1) -- cycle
 ;
\node at (6.85,-0.40){$xz$-projection};
\node[blue] at (6.65,3.85){$2$};
\node at (6.85,2.6){$xy$-projection};
\end{tikzpicture}
        \caption{
    The example curve $\chi = V(y^2-x^3-4x^2-8t^4x)$ in Figure \ref{fig-ellipticrational} inside of the modification constructed in Figure \ref{modifying_along_tropical_line}. We have actually already seen its $xz$-projection - it is $\trop(-x^3-4 xy+y^2-8t^4x)$ and is on the right of Figure \ref{fig-ellipticrational}.}
        \label{fig:running-inside-modification}
    \end{figure}
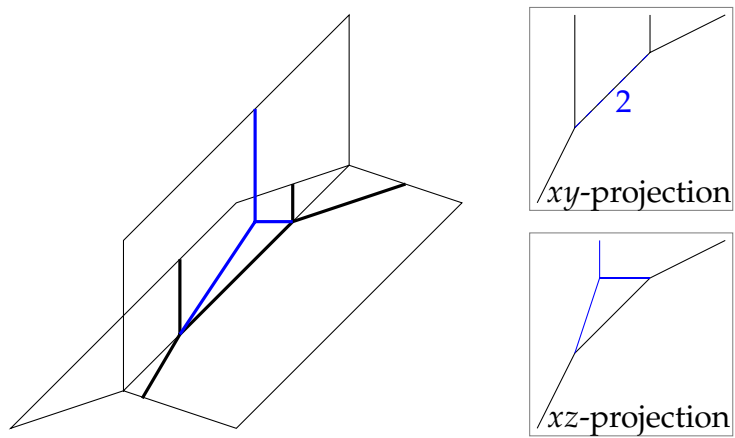
Finally, if we have identified a re-embedding polynomial constructed the appropriate tropical curve inside the modification, how can we be sure that the resulting tropical curve is actually faithful?

The answer lies in the multiplicities of the tropical curve \cite{MS2015}.
\begin{lemma}[\cite{BPR2014}, Theorem 5.24]\label{lem: cartify_faithful}
An embedded tropical curve is faithful on a subgraph $\Sigma$ of the Berkovich space $\chi^{an}$  if it contains a copy of $\Sigma$ whose edges  have tropical multiplicity one and all vertices are three-valent, i.e. have three incident edges or legs.
\end{lemma}

\section{Faithful tropicalizations of building blocks}\label{sec_local_faithful}
In this section we find polynomials for faithful re-embeddings for curves that have genus at least one, i.e. the upper part of Table \ref{table_standard_trop_coeff_cond}. We proceed by carefully analyzing the tropical curve obtained in the re-embedding by analyzing its coordinate projections.

\begin{remark}[Re-embeddings for cycles and 2-thetas]
In \cite{CM2016}, re-embedding polynomials (see Definition \ref{notation}) for cycles have been constructed as $f(x,y) = y+ \sqrt{-\alpha_3}x$
and in \cite{CM2019}, re-embedding polynomials for 2-thetas were constructed as $f(x,y) = y- \sqrt{-\alpha_3\alpha_4\alpha_5}x + \sqrt{-\alpha_5}x^2$, where the $\alpha_i$ are the roots of the defining polynomial of the hyperelliptic curve.
\end{remark}

\begin{theorem}[Faithful re-embeddings of 3-thetas]\label{thm_3-thetas_faithful}
Let $\chi$ be a hyperelliptic curve given by a defining equation $g$ satisfying the coefficient conditions of being a 3-theta as in Table \ref{table_standard_trop_coeff_cond}. Then, the embedding of $\trop(\chi)$ into the modification of $\R^2$ at $y - (\sqrt{(\prod_{i = 3}^{7} \alpha_i)} x + \sqrt{(\alpha_5\alpha_6\alpha_7)} x^2 - \sqrt{\alpha_{7}}x^3)$ given by $I_{g,f}$ is faithful on the minimal Berkovich skeleton.
\end{theorem}
Depending on the valuations of the coefficients, several cases have to be considered. The case-by-case analysis is summed up in the following proof. The tables, code and sub\-procedures providing further details of this proof will be given throughout the section.
\begin{proof}
For the remainder of this section, let $\alpha_i = \beta_i^2$ for $i = 3, ..., 6$ and $\alpha_{7} = -\beta_{7}^2$, so that the re-embedding may be equivalently given as $y - (\prod_3^7\beta_ix+\beta_5\beta_6\beta_7x^2-\beta_7x^3).$
We have $\init(\alpha_3) = \init(\alpha_4)$ and $\init(\alpha_5) = \init(\alpha_6)$ by Table \ref{table_standard_trop_coeff_cond}. Let $d_{34} = \val(\alpha_3-\alpha_4)$ and $d_{56}=\val(\alpha_5-\alpha_6)$. We control these differences by defining new variables $ \beta_{34} := \beta_3-\beta_4$ and $ \beta_{56} := \beta_5-\beta_6$, and substituting $\alpha_3:= (\beta_4 + \beta_{34})^2$ and $ \alpha_5:= (\beta_6 + \beta_{56})^2$. We obtain
$-\val(\beta_{34}) = d_{34}+\val(\alpha_4)/2= -\val(\alpha_3-\alpha_4)+\val(\beta_4)=\omega_{34} $
and
$-\val(\beta_{56}) = d_{56}+\val(\alpha_6)/2 = -\val(\alpha_5-\alpha_6)+\val(\beta_6)=\omega_{56}. $

Substituting the new variables into the Equation \eqref{eq_defining} of the hyperelliptic curve, we obtain
$f(x,y) = y -  \beta_4(\beta_{34}+\beta_4)\beta_6(\beta_6+\beta_{56})\beta_7x - \beta_6(\beta_6+\beta_{56})\beta_7x^2 + \beta_7x^3.$
The weight vector $\underline{u}\in\R^6$ encoding the valuation of the six parameters equals
$$\underline{u}=\val(\beta_7,\beta_6,\beta_{56},\beta_4,\beta_{34},\beta_2)=(\frac{\omega_7}{2},\frac{\omega_6}{2}, d_{56}-\frac{\omega_6}{2},\frac{\omega_4}{2}, d_{34}-\frac{\omega_4}{2},\frac{\omega_2}{2}).$$
Using the valuation conditions from Table \ref{table_standard_trop_coeff_cond}, the valuations satisfy the following inequalities:
\begin{align}\label{eq_defining_inequalities_3theta}
u_7 > u_6, \hspace{0.2cm}
u_6 > u_{56}, \hspace{0.2cm}
u_6 > u_4, \hspace{0.2cm}
u_4 > u_{34},  \hspace{0.2cm}\mathrm{ and } \hspace{0.2cm}
u_4 > u_2.
\end{align}
These inequalities span a polyhedral cone $\mathcal{C} \subset \R^6$. Its proper faces correspond to other types of curves. We subdivide the cone further using the data of the $xz$- and the $yz$-projections of $I_{g,f}$, computed in \textsc{Oscar} \cite{Oscar}  via elimination. For the sake of brevity, we only supply the full computations for the maximal cones of this subdivision here. Supplementary material and further analysis of the lower dimensional cones can be found at \url{https://victoriaschleis.github.io/bookchapter}.

For both projections, we list the leading terms that can appear for all monomials in the projection. For the $xz$-projection, this gives rise to 81 cones, of which 16 are maximal, defined by the inequalities in Table  \ref{defining_conditions_xz}, and give the corresponding leading terms of the $xz$-projection in Table \ref{table_xz_analysis}.

Using the $yz$-projection, we subdivide our cones further as given in Remark \ref{lem tropical dominance conditions}. We do this computationally by taking the common refinement of the two fans spanned by these respective collections of cones. We  list the leading terms of the  $yz$-projection in  Table \ref{leading_terms_coneI}.

We proceed by giving a precise analysis of the Newton subdivision inside of each cone in Lemma \ref{3lem_existence_cones_xz} for the $xz$-projection and in Proposition \ref{prop: faithful_yz} for the $yz$-projection and list them in Table \ref{table_tropical_curves_xz} and Figure  \ref{table_newton_subdivs_yz} respectively.

Combining the information of these two projections with the naive tropicalization of the curve from Section \ref{sec_reembeddings_and_blocks} and with Lemma \ref{lem_proj_lemma_genus_n}, we obtain multiplicity one on all edges of the $3$-theta and thus, using Lemma \ref{lem: cartify_faithful}, faithfulness on the minimal Berkovich skeleton.
\end{proof}

\subsection{Analysis of the $xz$- projection}\label{sec_xz-projection}
We begin our analysis of the re-embedding $I_{g,f}$ by investigating the $xz$-projection of  Trop$(V(I_{g,f}))$. Computationally, this is done by elimination of $y$, using the \texttt{eliminate}-command in \textsc{OSCAR} :

\begin{minted}[breaklines]{jlcon}
julia> S,(b2,b34,b4,b56,b6,b7)=QQ["b2", "b34", "b4", "b56", "b6", "b7"]
(Multivariate polynomial ring in 6 variables over QQ, QQMPolyRingElem[b2, b34, b4, b56, b6, b7])

julia> R,(x, y, z)=S["x", "y", "z"]
(Multivariate polynomial ring in 3 variables over multivariate polynomial ring, AbstractAlgebra.Generic.MPoly{QQMPolyRingElem}[x, y, z])

julia> g = y^2-x*(x-b2^2)*(x-(b34+b4)^2)*(x-(b4)^2)*(x-(b6+b56)^2) *(x-(b6)^2)*(x+(b7)^2);

julia> f= y - b4*(b34 + b4)*b6*(b6 + b56)*b7*x + b6*(b6 + b56)*b7*x^2 - b7*x^3;

julia> Igf = ideal(R,[g,f]);

julia> I_xz = eliminate(Igf,[y]);
\end{minted}

The resulting projection is alternatively given by the polynomial $g_{xz} (x,z) :=  g(x, z + (\prod_{3}^{7} \beta_i) x + \beta_5\beta_6\beta_7 x^2 - \beta_{7}x^3).$

\begin{lemma}\label{3lem_existence_cones_xz}
The conditions on the valuations of the coefficients for 3-thetas (see \eqref{eq_defining_inequalities_3theta}) describe a cone $\mathcal{C}$ which is subdivided into 81 cones by inequalities of the coefficients of the $xz$-projection.  16 of these are maximal corresponding to the rows in Table \ref{defining_conditions_xz}, although only seven distinct Newton subdivisions arise from them.
\end{lemma}
\begin{proof}
To analyze the Newton subdivision of $g_{xz}$, we compute all possible leading terms of all monomials in the $xz$-projection. Except for $x^5$ and $x^3$, all monomials
admit the same leading terms for any admissible choice of ordering on the $\beta_i$ using \eqref{eq_defining_inequalities_3theta}. Omitting constant coefficients, for $x^5$, the possible leading terms are $\beta_6^4$, $-\beta_{56}^2\beta_7^2$, and $-\beta_4^2\beta_7^2$. For $x^3$, the possible leading terms are $\beta_2^2\beta_6^4\beta_7^2$, $\beta_{34}^2\beta_6^4\beta_7^2$, and $\beta_4^4\beta_6^2\beta_7^2$. For the remaining coefficients, we refer to Table \ref{table_xz_analysis}.

  \begin{table}[htbp]
\centering
\begin{tabular}{ccccccccc}
\toprule 
$x^7, z^2$ &$x^6$ &$x^4$ & $x^2$ & $x$ &$x^3z$ &$x^2z$ &$xz$\\
\midrule 
$\beta_6^2$&  $\beta_4^2\beta_6^2\beta_7^2$ & $\beta_2^2\beta_4^2\beta_6^4\beta_7^2$& $\beta_2^2\beta_4^4\beta_6^4\beta_7^2$&$\beta_7$&$\beta_6^2\beta_7$&$\beta_4^2\beta_6^2\beta_7$\\
\bottomrule
\end{tabular}
\label{table_xz_analysis}
\caption{Valuations of the $xz$-projection of the modification ideal $I_{g,f}$}
\end{table}

From the computed initial forms we can determine the induced subdivision of the cone $\mathcal{C}$ by the occurring possible leading terms. The inequalities are as follows:

  \begin{table}[htbp]
\centering
\begin{tabular}{ccc||ccc}
\toprule \textbf{Cone} & \textbf{Condition 1} & \textbf{Condition 2} & \textbf{Cone}& \textbf{Condition 3} & \textbf{Condition 4}\\
\midrule $\mathcal{C}_1$ & $2\omega_6 < \omega_{56}+\omega_7$ & $\omega_{56}<\omega_{4}$ & $\mathcal{C}_A$ & $2\omega_4 < \omega_{34}+\omega_6$ & $\omega_{34}<\omega_2$\\
$\mathcal{C}_2$ & $2\omega_6 < \omega_{4}+\omega_7$ & $\omega_{4}<\omega_{56}$ & $\mathcal{C}_B$ & $2\omega_4 < \omega_{2}+\omega_6$ & $\omega_{2}<\omega_{34}$\\
$\mathcal{C}_3$ & $\omega_{56}+\omega_7<2\omega_6$ & $\omega_{56}<\omega_{4}$ & $\mathcal{C}_C$ & $ \omega_{34}+\omega_6<2\omega_4$& $\omega_{34}<\omega_{2}$\\
$\mathcal{C}_4$ & $\omega_{4}+\omega_7<2\omega_6$ & $\omega_{4}<\omega_{56}$ & $\mathcal{C}_D$ & $ \omega_{2}+\omega_6<2\omega_4$& $\omega_{2}<\omega_{34}$\\ 
\bottomrule
\end{tabular} \caption{Coefficient conditions for cones in the $xz$-projection}
 \label{defining_conditions_xz}
\end{table}

Now we can assign each of the possible leading terms of $x^5$ and $x^3$ we computed earlier to a cone: for $x^5$, $\beta_6^4 $ leads for points in $\mathcal{C}_1$ and $\mathcal{C}_2$, $-\beta_{56}^2\beta_7^2$ for  $\mathcal{C}_3$  and $-\beta_4^2\beta_7^2$ for $\mathcal{C}_4$. For $x^3$, the leading terms are $\beta_2^2\beta_6^4\beta_7^2$ for points in $\mathcal{C}_D$, $\beta_{34}^2\beta_6^4\beta_7^2$ for $\mathcal{C}_C$, and $\beta_4^4\beta_6^2\beta_7^2$ for $\mathcal{C}_A$ and $\mathcal{C}_B$.

There are seven possible combinatorial types of tropical curves that can arise in the $xz$-projection of the re-embedding, corresponding to the cones determined above. We give them and their corresponding Newton subdivisions in Table \ref{table_tropical_curves_xz}. All Newton subdivisions share a common part, given in Figure \ref{rem_general_xz_subdiv}, and one that is distinct for each type, given as a subdivision of $\mathcal{P}$ in Table \ref{table_tropical_curves_xz}.
\end{proof}

\begin{figure}
\includegraphics{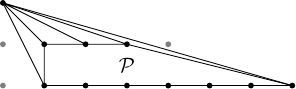}
\caption{General $xz$-subdivision}
\label{rem_general_xz_subdiv}
\end{figure}

Source codes, examples and material for the analysis of the remaining nonmaximal cones can be found at \url{https://victoriaschleis.github.io/bookchapter}. Further, an example for each maximal combinatorial type is given there.

The following Table \ref{table_tropical_curves_xz} is color-coded in accordance with the later analysis for the $yz$-projections, and the tropical curves are re-drawn versions of the \textsc{Oscar} \cite{Oscar} output to include colors.

\begin{longtable}{ c |c |c }
    \makecell{$\mathcal{C}_{1A},\mathcal{C}_{1B},\mathcal{C}_{2A}$ and $\mathcal{C}_{2B}$} & $\mathcal{C}_{1C}, \mathcal{C}_{1D}$ and $\mathcal{C}_{2C}$&$\mathcal{C}_{2D}$ \\
    \begin{minipage}{4cm} \centering \vspace{0.1cm}\includegraphics[scale=0.7]{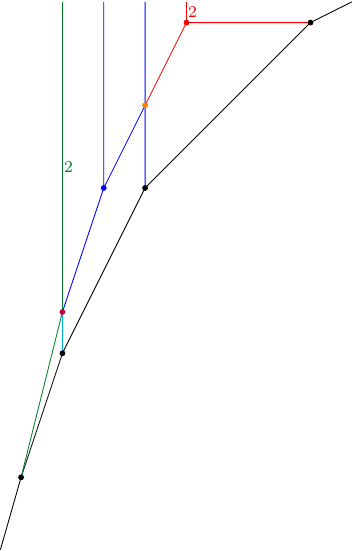}\vspace{0.1cm} 
    \end{minipage} & \begin{minipage}{4cm} \centering \vspace{0.1cm}\includegraphics[scale=0.7]{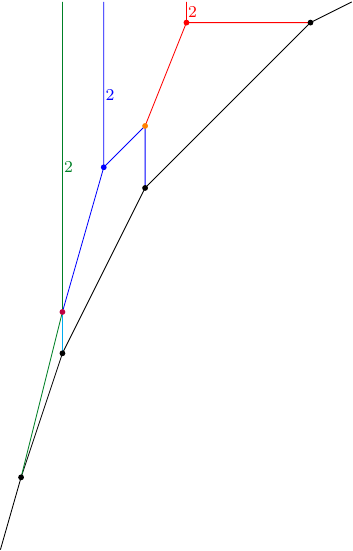}\vspace{0.1cm}
    \end{minipage} &  \begin{minipage}{4cm} \centering \vspace{0.1cm}\includegraphics[scale=0.7]{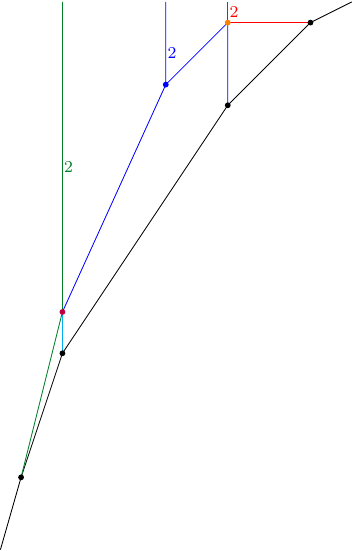}\vspace{0.1cm}
    \end{minipage} \\
    \begin{minipage}{4.25cm} \centering \vspace{0.1cm}\includegraphics{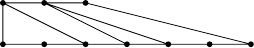}\vspace{0.1cm}
    \end{minipage}&\begin{minipage}{4.25cm} \centering \vspace{0.1cm}\includegraphics{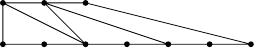}\vspace{0.1cm}
    \end{minipage} &\begin{minipage}{4.25cm} \centering \vspace{0.1cm}\includegraphics{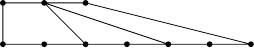}\vspace{0.1cm}
    \end{minipage}\\  \midrule
    $\mathcal{C}_{3A}$ and $\mathcal{C}_{3B}$ & $\mathcal{C}_{3C}$ & $\mathcal{C}_{3D}$ \\
    \begin{minipage}{4cm} \centering \vspace{0.1cm}\includegraphics[scale=0.7]{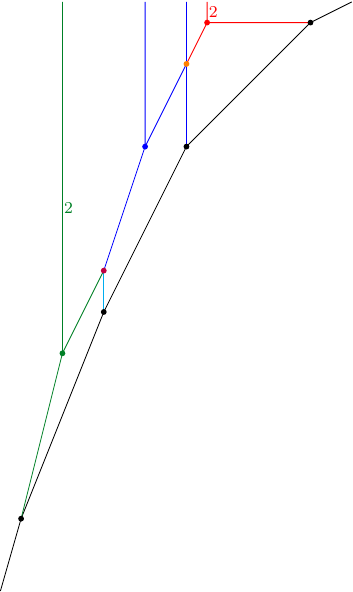}\vspace{0.1cm} 
    \end{minipage} & \begin{minipage}{4cm} \centering \vspace{0.1cm}\includegraphics[scale=0.7]{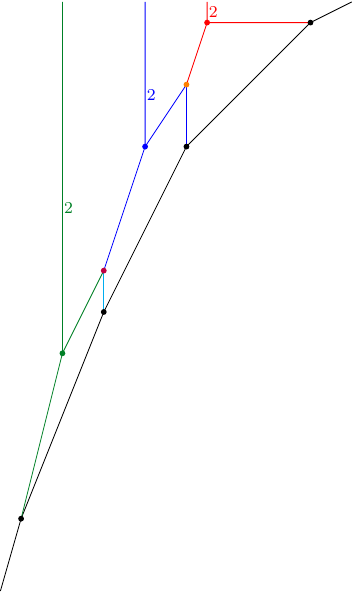}\vspace{0.1cm}
    \end{minipage} &  \begin{minipage}{4cm} \centering \vspace{0.1cm}\includegraphics[scale=0.7]{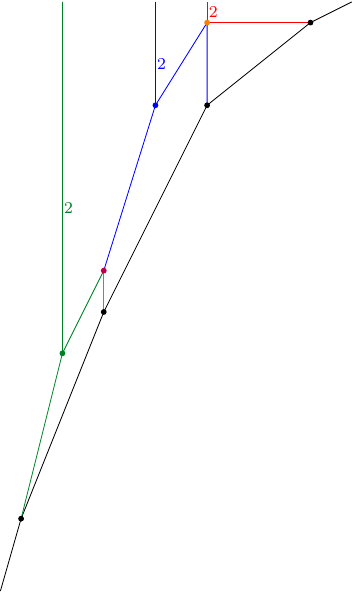}\vspace{0.1cm} 
    \end{minipage} \\   \begin{minipage}{4.25cm} \centering \vspace{0.1cm}\includegraphics{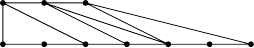}\vspace{0.1cm}
    \end{minipage}&\begin{minipage}{4.25cm} \centering \vspace{0.1cm}\includegraphics{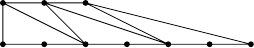}\vspace{0.1cm}
    \end{minipage} &\begin{minipage}{4.25cm} \centering \vspace{0.1cm}\includegraphics{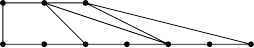}\vspace{0.1cm}
    \end{minipage}\\ \midrule
    & $\mathcal{C}_4$ & \\ & \begin{minipage}{4cm} \centering \vspace{0.1cm}\includegraphics[scale=0.7]{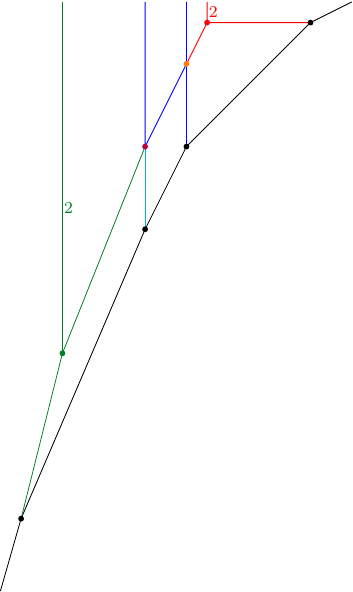}\vspace{0.1cm} 
    \end{minipage} & \\ &  \begin{minipage}{4.25cm} \centering \vspace{0.1cm}\includegraphics{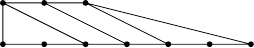}\vspace{0.1cm}
    \end{minipage}& \\

\caption{ Tropical curves associated to cones in the $xz$-projection. The associated Newton subdivisions given here are subdivisions of the trapeze $\mathcal{P}$ from Figure \ref{rem_general_xz_subdiv}.  }\label{table_tropical_curves_xz}
\end{longtable}

\subsection{Analysis of the $yz$-projection}
Since the analogous elimination computation for the $yz$-projection does not terminate within reasonable time (2 weeks), we proceed as follows. First, we investigate the structure of the possible coefficients and monomials in the $yz$-projection.

The $yz$-projection is given as the variety of a single defining polynomial of degree $7$ in $y$ and $z$. The coefficients of the monomials $y^iz^j$ are homogeneous polynomials $p_{y^iz^j}$ in $\beta_2,\beta_{34},\beta_{4},\beta_{56},\beta_{6}$ and $\beta_{7}$. The homogeneous degree of the $p_{y^iz^j}$ is 1 for monomials of degree $7$ and $(7-(i+j))\cdot 7$ otherwise. Using the defining equation of the cone of $3$-thetas \eqref{eq_defining_inequalities_3theta}, we obtain the following technical \emph{tropical dominance conditions} for coefficients:

\begin{remark}[Tropical dominance conditions]\label{lem tropical dominance conditions}
We note that by the structure of the equations, with the exception of $\beta_7$ all exponents in leading monomials are even, and if the exponent of $\beta_7$ is odd for one possible leading monomial in  $p_{y^iz^j}$, they are odd in $\beta_7$ for all leading monomials.
Thus, we can assume that inequalities on the valuations between  leading monomials are integral.
Further, there are minimal inequalities generating all inequalities, that are of the following form: $\omega_{i}<\omega_{j}$ and $\omega_{i}>\omega_{j}$ for pairs $(i,j)\in\{(2,34),(4,56)\}$, $2\omega_i < \omega_{j}+\omega_{k}$ and $2\omega_i > \omega_{j}+\omega_{k}$ for combinations $i,(j,k) \in \{(4,(2,6)),(4,(34,6)),(56,(2,6)),(56,(34,6)),(6,(4,7)),(6,(56,7))\}$.
We call a cone $\mathcal{C}_i$ in the subdivision provided by Table \ref{defining_conditions_xz} \emph{$\beta_{56}$-dominant} if $\omega_{56}<\omega_4$. In this case, only the coefficient conditions we already determined contribute.
Alternatively, we call $\mathcal{C}_i$ \emph{$\beta_4$-dominant} (i.e. $\omega_4<\omega_{56}$). In this case, we obtain an additional subdivision of the $yz$-projection, given by inequalities on $\omega_{56}$ in relation to $\omega_{34}$ and $\omega_2$, as now, potentially $\min(\omega_{34},\omega_2)<\omega_{56}$. This is not possible for $\beta_{56}$-dominant terms, as here, $\omega_{56}<\omega_4<\min(\omega_{34},\omega_2)$.
In particular, in addition to the conditions considered in Table \ref{defining_conditions_xz}, the following five inequalities contribute:

\small
\begin{longtable}{c | c | c | c}
    $\alpha_<$ & $2\beta_4 < \min(\beta_2,\beta_{34})+\beta_{56}$ & $\alpha_>$ & $2\beta_4 > \min(\beta_2,\beta_{34})+\beta_{56}$ \\
    $\beta_<$ & $\max(\beta_2,\beta_{34})+\beta_{56} < \min(\beta_2,\beta_{34})+\beta_{6}$& $\beta_>$ & $\max(\beta_2,\beta_{34})+\beta_{56} > \min(\beta_2,\beta_{34})+\beta_{6}$\\
    $\gamma_<$ & $\min(\beta_2,\beta_{34})+\beta_{4} < \max(\beta_2,\beta_{34})+\beta_{56}$ & $\gamma_>$ & $\min(\beta_2,\beta_{34})+\beta_{4} > \max(\beta_2,\beta_{34})+\beta_{56}$\\
    $\delta_<$ &  $\min(\beta_2,\beta_{34})+2\beta_{6} < 2\beta_4+\beta_{7}$ &$\delta_>$ &  $\min(\beta_2,\beta_{34})+2\beta_{6} > 2\beta_4+\beta_{7}$\\
    $\phi_<$ & $2\beta_4 < \beta_{56} + \beta_6$ &
    $\phi_>$ & $2\beta_4 > \beta_{56} + \beta_6$
\end{longtable}
\end{remark}
\normalsize

\begin{proposition}[Leading terms]
    The leading terms for the coefficients of monomials given in Table \ref{leading_terms_coneI} are correct and complete.
\end{proposition}
\begin{proof}
    We proceed in two steps. First, we consider monomials of degree  $\geq 5$.
    For monomials $y^iz^j$ of degree $\geq 5$, the $p_{y^iz^j}$ are of homogeneous degree $\leq 14$. By the structure of the equations, we then expect $p_{y^iz^j}$ to contain monomials that are only in terms of $\beta_{4},\beta_{56},\beta_{6}$ and $\beta_{7}$, and where the exponents in terms of $\beta_4$ and $\beta_{56}$ sum to less than or equal 4. Using the tropical dominance conditions from Remark \ref{lem tropical dominance conditions}, this implies that these are indeed the only possible leading terms for the monomial satisfying the defining inequality \eqref{eq_defining_inequalities_3theta}.  Thus, we can substitute integers for the coefficients $\beta_2$ and $\beta_{34}$ and compute the elimination only with respect to the remaining coefficients. This can be computed in \textsc{Oscar} \cite{Oscar}  using

\begin{minted}[breaklines]{jlcon}
julia> S,(b4,b56,b6,b7)=QQ["b4", "b56", "b6", "b7"]
(Multivariate polynomial ring in 4 variables over QQ, QQMPolyRingElem[b4,b56, b6, b7])

julia> R,(x, y, z)=S["x", "y", "z"]
(Multivariate polynomial ring in 3 variables over multivariate polynomial ring, AbstractAlgebra.Generic.MPoly{QQMPolyRingElem}[x, y, z])

julia> g = y^2-x*(x-1000^2)*(x-(1000000+b4)^2)*(x-(b4)^2)*(x-(b6+b56)^2)*(x-(b6)^2) *(x+(b7)^2);

julia> f = y - b4*(1000000 + b4)*b6*(b6 + b56)*b7*x + b6*(b6 + b56)*b7*x^2 - b7*x^3;

julia> Igf = ideal(R,[g,f]);

julia> I_xy = eliminate(Igf,[x]);
\end{minted}

     which terminates within two hours. This procedure turns the previously homogeneous polynomials into non-homogeneous ones. The leading terms are now in the coefficients of highest degree. Here, we can identify the actual leading terms using the tropical dominance conditions from Remark \ref{lem tropical dominance conditions} encoded in the following algorithm:

    \begin{algorithm}
     \label{alg_initials}
\SetAlgoLined
\KwData{A polynomial $p_{y^iz^j}$ obtained from the elimination}

\KwResult{A list of potential leading terms $lt$ satisfying the inequalities in \eqref{eq_defining_inequalities_3theta}}\;
 $h_{b_7}$ is the maximal exponent of $\beta_7$ in $p_{y^iz^j}$ potential\_monomials = list of monomials containing $\beta_7^{h_{b_7}}$\;
 $h_{b_6}$ is the maximal exponent of $\beta_6$ in potential\_monomials\;
  \While{$h_{b_6}<$ expected\_degree and $h_{b_7}$, $h_{b_{56}}$, $h_{b_4}>0$}
{ curr\_monomial = monomial of maximal exponent $h_{b_4}$ of $\beta_4$ (resp. $h_{b_{56}}$ of $\beta_{56}$) among monomials containing $\beta_6^{h_{b_6}}\beta_7^{h_{b_7}}$\;
         \eIf{deg(current\_monomial) = expected\_degree} {$lt$ = $lt\cup$ current\_monomial}
{ $lt=lt\cup$cross\_verification(current\_monomial)}
 Adjust $h_{b_7}$, $h_{b_6}$, $h_{b_{56}}$, and $h_{b_4}$ according to Remark \ref{lem tropical dominance conditions}
}

\caption{leading terms}
 \end{algorithm}

    Now, we proceed with monomials of lower degree. Here, knowing the elimination results in coefficients $\beta_{4},\beta_{56},\beta_{6}$ and $\beta_{7}$ is not sufficient, as the leading terms  contain coefficients $\beta_2$ and $\beta_{34}$.

    To tackle this, we now compute the analogous eliminations, where instead of substituting integers for $\beta_2$ and $\beta_{34}$, we substitute integers for one of $\beta_{2}$ and $\beta_{34}$ and one of $\beta_{4}$ and $\beta_{56}$ respectively, and call the outcomes of these elimination computations \emph{cross-verification polynomials}.

    Since $\beta_4,\beta_{56},\beta_6$ and $\beta_7$ remain the dominant coefficients, we first identify the potential monomials again as in Algorithm \ref{alg_initials}, remembering the gap $k$ in exponents. As we are in a maximal cone, we can assume $\omega_{34}>\omega_2$ or vice versa. Thus, in the cross-verification polynomials, we search for a term containing the correct amount of $\beta_4$ or $\beta_{56}$ respectively, and the maximal exponent of $\beta_2$ or $\beta_{34}$ less than or equal to $k$. We verify the existence of such a term using the remaining cross-verification polynomials.

    For $\beta_4$-dominant terms, in accordance with Remark \ref{lem tropical dominance conditions} we repeat the procedure and identify further leading terms by searching for terms containing maximal amounts of $\beta_2$ and $\beta_{34}$ instead of $\beta_{56}$. For mixed terms, we again use the cross-verification process. This computes all leading terms in five or less variables.

    For the final case of leading terms involving all six variables, we observe that the cross-verification algorithm yields two symmetric outputs of a degree smaller than the expected degree. We combine the two leading terms computed into a single leading term missing no monomials.

    We supply code computing all possible leading terms at \url{https://victoriaschleis.github.io/bookchapter}.
\end{proof}

\small
\begin{longtable}{ c  c  c   c  }
\toprule \textbf{Monomial} & \textbf{Leading Terms} & \textbf{Weights} & \makecell{\textbf{Cone }}\\ \midrule
\midrule

\makecell{$y^7, y^6z, y^5z^2,$\\$ y^4z^3,y^3z^4, $\\$ y^2z^5, yz^6,z^7$ }& \makecell{1}& $0$ & all\\ \midrule\midrule

$y^6$ & \makecell{$\beta_4^2\beta_6^2\beta_7^3$ \\{\color{black}$\beta_{56}^2\beta_6^2\beta_7^3$}  } & \makecell{$\omega_4+\omega_6+3\omega_7/2$ \\$\omega_{56}+\omega_6+3\omega_7/2$ } & \makecell{$(2,4)$\\$(1,3)$} \\ \midrule

$y^5z$ & \makecell{$\beta_4^2\beta_7^5$ \\$\beta_{56}^2\beta_7^5$} & \makecell{ $ \omega_4 + 5 \omega_7/2$\\ $ \omega_{56} + 5 \omega_7/2$} & \makecell{$(2,4)$\\$(1,3)$}\\ \midrule

$y^4z^2$ & $\beta_6^2\beta_7^5$ & $\omega_6+5\omega_7/2$ & all \\ \midrule

\makecell{$y^3z^3,y^2z^4,$\\$yz^5,z^6$ }& \makecell{$\beta_7^7$  }& $7\omega_7/2$ & all \\ \midrule \midrule

$y^5$ & \makecell{  $\beta_{56}^2\beta_6^8\beta_7^4 $  \\ $\beta_4^6\beta_7^8$ \\ $\beta_{56}^6\beta_7^8$ } & \makecell{$\omega_{56}+4\omega_6+2\omega_7 $  \\ $3\omega_4+4\omega_7$ \\ $3\omega_{56}+4\omega_7$ }& \makecell{$(1,2)$\\$(4)$\\$(3)$}
\\ \midrule

$y^4z$ & \makecell{ $\beta_{56}^2\beta_6^{6}\beta_7^6$ \\ $\beta_{4}^4\beta_6^2\beta_7^8$ \\ $\beta_{56}^4\beta_6^2\beta_7^8$   } & \makecell{ $\omega_{56}+3\omega_6+3\omega_7 $  \\ $2\omega_4+\omega_6+4\omega_7$ \\ $2\omega_{56}+\omega_6+4\omega_7$ } & \makecell{($1,2$) \\ ($4$) \\ ($3$) } \\ \midrule

\makecell{$y^3z^2, y^2z^3,$\\$ yz^4, z^5$} & \makecell{  $\beta_4^2\beta_6^4\beta_7^8$ \\ $\beta_{56}^2\beta_6^4\beta_7^8$  } & \makecell{ $\omega_4+2\omega_6+4\omega_7$ \\ $\omega_{56}+2\omega_6+4\omega_7$} & \makecell{$(2,4)$\\$(1,3)$}

\\ \midrule \midrule

$y^4$ & \makecell{$\beta_{34}^2\beta_{56}^2\beta_6^{10}\beta_7^7$ \\ $\beta_2^2\beta_{56}^2\beta_6^{10}\beta_7^7$\\$\beta_{34}^2\beta_{56}^4\beta_6^6\beta_7^9$ \\ $\beta_2^2\beta_{56}^4\beta_6^6\beta_7^9$ \\ $\beta_4^4\beta_{56}^2\beta_6^6\beta_7^9$  }&  \makecell{$\omega_{34}+\omega_{56}+5\omega_6+7/2\omega_7$ \\ $\omega_2+\omega_{56}+5\omega_6+7/2\omega_7$\\$\omega_{34}+2\omega_{56}+3\omega_6+9/2\omega_7$ \\ $\omega_2+2\omega_{56}+3\omega_6+9/2\omega_7$ \\ $2\omega_4+\omega_{56}+3\omega_6+9/2\omega_7 $ }& \makecell{($1A_{\delta_{<}}$,$1C_{\delta_{<}}$,$2A_{\delta_{<}}$,$2C_{\delta_{<}}$)  \\ ($1B_{\delta_{<}}$,$1D_{\delta_{<}}$,$2B_{\delta_{<}}$,$2D_{\delta_{<}}$) \\($3A_{\alpha_{>}}$,$3C_{\alpha_{>}}$,$4A_{\alpha_{>}}$,$4C_{\alpha_{>}}$) \\ ($3B_{\alpha_{>}}$,$3D_{\alpha_{>}}$,$4B_{\alpha_{>}}$,$4D_{\alpha_{>}}$)\\ ($\alpha_{<},\delta_{>}$)} \\ \midrule

$y^3z$ & \makecell{ $\beta_4^2\beta_{56}^2\beta_6^8\beta_7^9$ }
 & \makecell{$ \omega_4+\omega_{56}+4\omega_6+9/2\omega_7 $} & \makecell{all}\\ \midrule

$y^2z^2, yz^3, z^4$ & \makecell{$\beta_{56}^2\beta_6^{10}\beta_7^9$ }&\makecell{ $\omega_{56}+5\omega_6+9\omega_7/2$}& \makecell{all}

\\ \midrule \midrule

 $y^3$ & \makecell{ $\beta_2^4\beta_{56}^2\beta_6^{12}\beta_7^{10}$\\$\beta_{34}^4\beta_{56}^2\beta_6^{12}\beta_7^{10}$\\$\beta_2^2\beta_4^4\beta_{56}^2\beta_6^{10}\beta_7^{10}$\\$\beta_{34}^2\beta_4^4\beta_{56}^2\beta_6^{10}\beta_7^{10}$} & \makecell{ $2\omega_2+\omega_{56}+6\omega_6+5\omega_7$\\$2\omega_{34}+\omega_{56}+6\omega_6+5\omega_7$\\$\omega_2+2\omega_4+\omega_{56}+5\omega_6+5\omega_7$\\$\omega_{34}+2\omega_4+\omega_{56}+5\omega_6+5\omega_7$ } & \makecell{($D$)\\($C$)\\($B$)\\($A$)} \\ \midrule
 $y^2z, yz^2,z^3$ & \makecell{ $\beta_2^2\beta_4^2\beta_{56}^2\beta_6^{12}\beta_7^{10}$\\$\beta_{34}^2\beta_4^2\beta_{56}^2\beta_6^{12}\beta_7^{10}$} & \makecell{ $\omega_2+\omega_4+\omega_{56}+6\omega_6+5\omega_7$\\$\omega_{34}+\omega_4+\omega_{56}+6\omega_6+5\omega_7$ } & \makecell{($B,D$)\\($A,C$)} \\ \midrule\midrule
 $y^2$ & \makecell{ $\beta_2^2\beta_{34}^2\beta_4^4\beta_{56}^2\beta_6^{14}\beta_7^{11}$\\$\beta_2^4\beta_4^4\beta_{56}^4\beta_6^{12}\beta_7^{11}$\\$\beta_{34}^4\beta_4^4\beta_{56}^4\beta_6^{12}\beta_7^{11}$\\$\beta_2^2\beta_4^8\beta_{56}^2\beta_6^{12}\beta_7^{11}$\\$\beta_{34}^2\beta_4^8\beta_{56}^2\beta_6^{12}\beta_7^{11}$} & \makecell{$\omega_2+\omega_{34}+2\omega_4+\omega_{56}+7\omega_6+11/2\omega_7$\\$2\omega_2+2\omega_4+2\omega_{56}+6\omega_6+11/2\omega_7$\\$2\omega_{34}+2\omega_4+2\omega_{56}+6\omega_6+11/2\omega_7$\\$\omega_2+4\omega_4+\omega_{56}+6\omega_6+11/2\omega_7$\\$\omega_{34}+4\omega_4+\omega_{56}+6\omega_6+11/2\omega_7$ } & \makecell{ ($D_{\beta_{>}}$,$C_{\beta_{>}}$) \\ ($B_{\alpha_{>}}$,$D_{\beta_{<}}$)\\ ($A_{\alpha_{>}}$,$C_{\beta_{<}}$) \\ ($B_{\alpha_{<}}$) \\($A_{\alpha_{<}}$) } \\ \midrule
 $yz, z^2$ & \makecell{ $\beta_{34}^2\beta_4^6\beta_{56}^2\beta_6^{14}\beta_7^{11}$} & \makecell{$\omega_{34}+3\omega_4+\omega_{56}+7\omega_6+11/2\omega_7$ } & \makecell{all} \\
  \midrule \midrule
 $y,z$ & \makecell{ $\beta_2^2\beta_{34}^2\beta_4^8\beta_{56}^2\beta_6^{16}\beta_7^{12}$\\$\beta_2^4\beta_4^8\beta_{56}^4\beta_6^{14}\beta_7^{12}$\\$\beta_{34}^4\beta_4^8\beta_{56}^4\beta_6^{14}\beta_7^{12}$\\$\beta_2^2\beta_{34}^2\beta_4^{12}\beta_6^{14}\beta_7^{12}$\\$\beta_2^4\beta_4^{10}\beta_{56}^2\beta_6^{14}\beta_7^{12}$\\$\beta_{34}^4\beta_4^{10}\beta_{56}^2\beta_6^{14}\beta_7^{12}$} & \makecell{ $\omega_2+\omega_{34}+4\omega_4+\omega_{56}+8\omega_6+6\omega_7$\\$2\omega_2+4\omega_4+2\omega_{56}+7\omega_6+6\omega_7$\\$2\omega_{34}+4\omega_4+2\omega_{56}+7\omega_6+6\omega_7$\\$\omega_2+\omega_{34}+6\omega_4+7\omega_6+6\omega_7$\\$2\omega_2+5\omega_4+\omega_{56}+7\omega_6+6\omega_7$\\$2\omega_{34}+5\omega_4+\omega_{56}+7\omega_6+6\omega_7$} & \makecell{($\phi_{>}\gamma_<$)\\ ($1B,1D,3B,3D$)\\($1A,1C,3A,3C$)\\($2_{\phi_<\gamma_>}$, $4_{\phi_<\gamma_>}$)\\($2B_{\gamma_{<}}$,$2D_{\gamma_{<}}$,$4B_{\gamma_{<}}$,$4D_{\gamma_{<}}$)\\($2A_{\gamma_{<}}$,$2C_{\gamma_{<}}$,$4A_{\gamma_{<}}$,$4C_{\gamma_{<}}$)} \\ \bottomrule
\caption{Leading terms for monomials in the $yz$-projection. }\label{leading_terms_coneI}

\end{longtable}
\normalsize
\begin{wrapfigure}{r}{7cm}\label{picture_reembedding}
\includegraphics{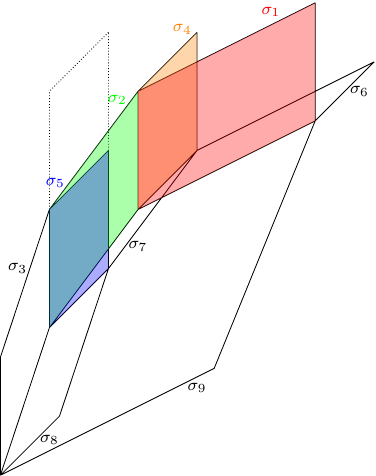}
\caption{The genereric re-embedding for $3$-thetas.}
\end{wrapfigure}

To analyze the projections, we to provide a geometric description of the re-embedding, as parallel cells need to be carefully analyzed.
For the re-embedding $f$ given in Theorem \ref{thm_3-thetas_faithful}, the tropical hypersurface Trop($V(z-f)$) consists of nine 2-dimensional cells, defined by the following systems of inequalities:

$\sigma_1 = \{ Z = X+A \geq Y, X \leq A-B \}$, 

$\sigma_2 = \{ Z = 2X + B \geq Y, B-C \geq X \geq A-B\}$, 

$\sigma_3 = \{ Z = 3X+C \geq Y, X \geq B-C \}$,

$\sigma_4 = \{  Z, Y \leq 2A-B, X = A-B\}$,

$\sigma_5 = \{ Z, Y \leq 3B -2C, X = B-C \}$,

$\sigma_6 = \{ Y = X+A \geq Z, X \leq A-B \}$,

$\sigma_7 = \{ Y = 2X+B \geq Z, A-B \leq X \leq B-C\}$,

$\sigma_8 = \{ Y = 3X+C \geq Z, B-C \leq X \}$,

$\sigma_9 = \{ Z = Y  \geq X+A,2X+B,3X+C
 \}$.

\begin{proposition}\label{prop: faithful_yz}
Using the re-embedding of Theorem \ref{thm_3-thetas_faithful}, all edges of the re-embedded tropical curve $\Trop(g)$ have tropical multiplicity one, i.e. the re-embedding yields a faithful tropicalization.
\end{proposition}
\begin{proof}

 We observe that in the $yz$-projection, the pairs of cells $\sigma_1$ and $\sigma_4$, $\sigma_1$ and $\sigma_5$, $\sigma_2$ and $\sigma_5$, and $\sigma_4$ and $\sigma_5$ overlap in parts. Further, the behaviour on these cells cannot be directly observed on the $xz$-projection. Thus, we expect some of the edges in the projections to have higher multiplicities, and need to carefully analyze them to understand the multiplicities and structure of the re-embedded curve. All in all, there are four types of such irregularities that can occur:

\begin{enumerate}
\item Some multiplicity one edges or legs inherit higher multiplicities in the $yz$-tropica\-lization due to the push-forward formula for multiplicities, see \cite{Tev07}. Multiplicities inherited in this way are the black multiplicities in Figure \ref{table_newton_subdivs_yz}.
\end{enumerate}
On the images of the overlapping cells $\sigma_1$, $\sigma_2$, $\sigma_4$ or $\sigma_5$, the following further behaviour can occur:
\begin{enumerate}
    \item[(ii)] Parallelograms inside the Newton subdivision may indicate either a vertex (red in the subdivision) or not (marked in blue), depending on whether two edges or legs of $I_{g,f}$ that intersect in the $yz$-projection intersect in $I_{g,f}$ or not.
    \item[(iii)] Two edges or legs overlap in the $yz$-tropicalization and their multiplicities get added accordingly, indicated by dashed edges and sums of multiplicities on the curve in Figure \ref{table_newton_subdivs_yz}.
    \item[(iv)] Vertices lie in the relative interiors of edges or on another vertex. The vertices for which this can happen are the purple and the orange vertices in Table \ref{table_tropical_curves_xz} and Figure \ref{table_newton_subdivs_yz}.
\end{enumerate}
We resolve these issues by investigating the corresponding $xz$-projection for each case and combining the information from both projections. A color-coded analysis of each individual representative is given in Table \ref{table_tropical_curves_xz} for the $xz$-projection and at \url{https://victoriaschleis.github.io/bookchapter} for the $yz$-projection.Edges that are the same color in the projections correspond to each other. See Example \ref{ex: refined modification}, where a three dimensional reconstruction is carried out in full for the cone $\mathcal{C}_0$. 
\end{proof}

\begin{figure}
\begin{longtable}{c c}
    \includegraphics[scale = 0.85]{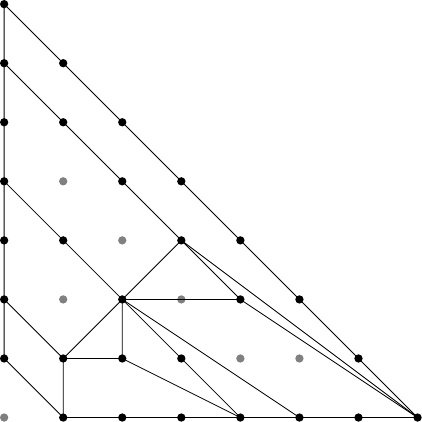} & \includegraphics[scale = 0.5]{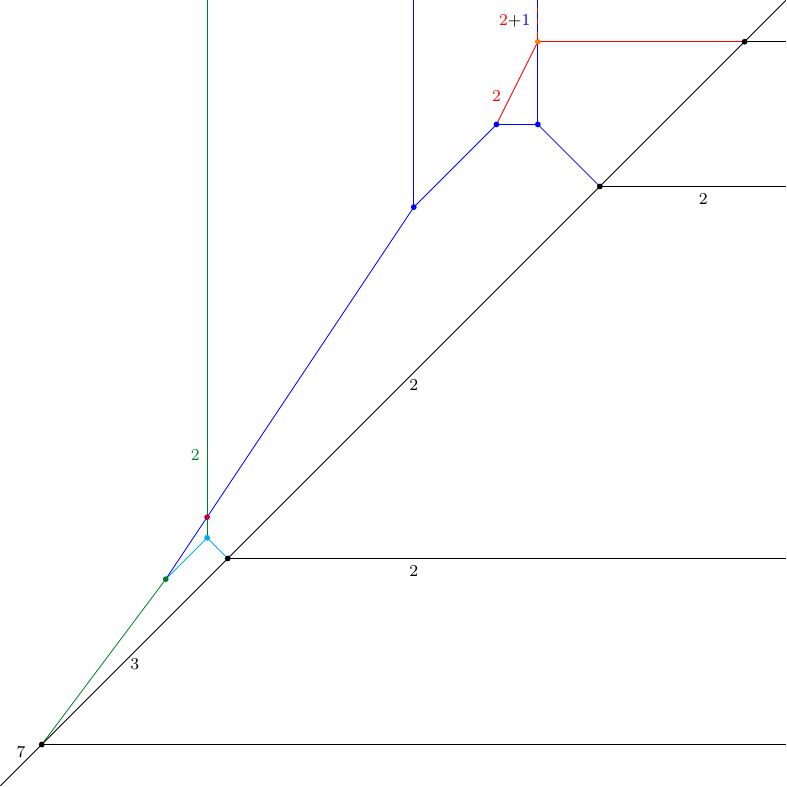}\\
\end{longtable}
    \caption{The $yz$-projection in Example \ref{ex: refined modification}. The tropical curve is a re-drawing of the \textsc{Oscar} \cite{Oscar} output to include colors.}
    \label{table_newton_subdivs_yz}
\end{figure}

\begin{example}\label{ex: refined modification}
    We give an example of a re-embedded curve in Cone $\mathcal{C}_{1A_{\alpha_<\delta_>}}$ in Figure \ref{table_newton_subdivs_yz}. The colors of edges and vertices coincide with the colors in the $xz$- and $yz$-projections in Table \ref{table_tropical_curves_xz} and Figure \ref{table_newton_subdivs_yz}.
\end{example}

\section{Gluing re-embeddings}\label{sec_globalizing}

In the previous section we have seen how to locally modify the tropical hyperelliptic curve according to specific building blocks of the associated Berkovich skeleton. This allows us to faithfully re-embed in the case of 3-thetas and 3-points. But what happens if the Berkovich skeleton consists of multiple building blocks?

In this section, we show how to glue the constructed re-embeddings of building blocks to a re-embedding that works for the whole curve.

\begin{remark}
 \label{lem_bridges_faithful}
First, let us assume that we are considering two building blocks $B_1$ and $B_2$ connected by a point connector $p$, and that we have a re-embedding that is faithful on both $B_1$ and $B_2$. Then, the re-embedding is already faithful on the minimal Berkovich skeleton.
Now, we assume that both building blocks are connected by a bridge $e_i$. The coefficient condition of bridges gives rise to a  dual edge of length one, as can be observed in Table \ref{table_standard_trop_coeff_cond}. Therefore, $e_i$ has tropical multiplicity one and  is faithful on every subgraph of the Berkovich skeleton that contains it. Thus, if the tropicalizations on both sides of the $e_i$ are faithful on respective subgraphs of the Berkovich skeleton connected by $e_i$ in the Berkovich skeleton, the whole curve is faithful on the union.
\end{remark}
The second piece of the gluing puzzle is the ability to shift re-embeddings: The modifications given in Section \ref{sec_local_faithful} work for the first monomial in the building block being $x$, but multiple building blocks will necessarily start at different monomials.
The next proposition solves just that: multiplying the modification polynomial by the correct monomial at which the building block starts preserves faithfulness.

\begin{lemma}\label{prop_shifting_building_blocks}
Let $\chi$ be an irreducible hyperelliptic curve given by the defining equation
$g(x,y) := y^2 - h(x) = 0$ and let $\tilde{\chi}$ be another curve given by the defining equation $\tilde{g}(x,y):= y^2 - x^{2l+1}h(x) = 0.$ Further, let $f \in K[x,y]$ be a modification polynomial such that  the tropicalization induced by the ideal $I_{g,f} = \langle g, z-f \rangle \in K[x^{\pm}, y^{\pm}, z^{\pm}]$ is faithful. Then, for the modification polynomial $\tilde{f}(x,y) = y - x^{l-1}f \in K[x,y]$, the tropicalization induced by the ideal $I_{\tilde{g},\tilde{f}} = \langle \tilde{g}, z-\tilde{f} \rangle \in K[x^{\pm}, y^{\pm}, z^{\pm}]$ is faithful.
\end{lemma}

\begin{proof}
Let $\mathcal{P}$ be the Newton polytope of $f$ with the negative valuation of every monomial corresponding to a point attached as a last coordinate, and let $\Tilde{\mathcal{P}}$ be the extended Newton polytope of $\Tilde{f}$. Then, for every vertex $\tilde{v}$  in $\tilde{\mathcal{P}}$ corresponding to the monomial $m$ there is a vertex $v$ in $\mathcal{P}$ corresponding to $mx^{-l}$  with the same valuation. Thus, $\mathcal{P}$ and $\tilde{\mathcal{P}}$ yield congruent subdivisions on their respective Newton polytopes $\mathcal{N}$ and $\tilde{\mathcal{N}}$ under the coordinate shift $x^i \mapsto x^{i+l}$. Analogous computations hold for the projections of their respective re-embeddings. Further, lattice lengths stay constant under this shift. Thus, $\tilde{g}$ only has multiplicity one edges in the re-embedding along $\tilde{f}$ by construction and is thus faithful by Lemma \ref{lem: cartify_faithful}.
\end{proof}

Combining the results from Section \ref{sec_local_faithful} and the shifting from Lemma \ref{prop_shifting_building_blocks}, we obtain the re-embedding polynomials for building blocks in Table \ref{table_reembedding}.

  \begin{table}[htbp]
\centering
\begin{tabular}{ccc}
\toprule 
\makecell{\textbf{Subgraph}\\ \textbf{at} $x^i$}& \makecell{\textbf{Local}\\ \textbf{Berkovich} \\ \textbf{skeleton}} & \makecell{\textbf{Re-embedding polynomial}} \\ \midrule \midrule

\makecell{$3$-Theta }& \begin{minipage}{2.5cm}\centering

\begin{tikzpicture}
            \draw(1,1.5) arc (90:-90:0.5);

            \draw(0,0.5) arc (90:-90:-0.5);
            \draw(0,0.5) -- (1,0.5);
            \draw(0,1.5) -- (1,1.5);
            \draw(0,1.5) -- (0,0.5);
            \draw(1,1.5) -- (1,0.5);
            \fill (0,0.5) circle (1.5pt);
            \fill (0,1.5) circle (1.5pt);
            \fill (1,0.5) circle (1.5pt);
            \fill (1,1.5) circle (1.5pt);
          \end{tikzpicture}

\end{minipage}
&\makecell{
$y - (\sqrt{(\prod_{i = 3}^{7} \alpha_i)} x + \sqrt{(\alpha_5\alpha_6\alpha_7)} x^2 - \sqrt{\alpha_{7}}x^3)$}  \\ \midrule
\makecell{$2$-Theta }& \begin{minipage}{2.5cm}\centering

\begin{tikzpicture}
            \draw(0,1.5) arc (90:-90:0.5);

            \draw(0,0.5) arc (90:-90:-0.5);
            \draw(0,1.5) -- (0,0.5);
            \fill (0,0.5) circle (1.5pt);
            \fill (0,1.5) circle (1.5pt);
          \end{tikzpicture}

\end{minipage}
& \makecell{$y- \sqrt{x^{i-1}}\cdot(\sqrt{-\alpha_3\alpha_4\alpha_5}x - \sqrt{-\alpha_5}x^2$)  }  \\ \midrule
\makecell{Cycle}& \begin{minipage}{2.5cm}\centering

\begin{tikzpicture}
            \draw(0,1.5) arc (90:-90:0.5);

            \draw(0,0.5) arc (90:-90:-0.5);
          \end{tikzpicture}

\end{minipage}
& \makecell{ $ y- \sqrt{x^{i-1}}\cdot \sqrt{-\alpha_3}x$}
\\ \bottomrule\end{tabular}
\caption{Re-embedding polynomials for the building blocks. The $\alpha_i$ are the coefficients in the defining equation of the hyperelliptic curve $\chi$, see \eqref{eq_defining}, and the $i$ is the exponent of the first monomial contributing to the cycle or 2-theta, as in Table \ref{table_standard_trop_coeff_cond}. \label{table_reembedding}}\end{table}

Finally, we can combine all the pieces constructed in the last Sections to obtain our main result, Theorem \ref{thm-main}.
\begin{theorem}\label{thm_mainthm_point_connectors_minimal_berkovich}
Let $\chi$ be an irreducible hyperelliptic curve of genus three defined by an equation $g$ as in Equation \ref{eq_defining}. The embeddding of $\chi$ in the ideal given by modifications along hyperplanes generated by $z_i - f_i(x,y)$, where $f_i(x,y) = y - x^l(a_1x - a_2x^2 - \dots - a_{k_i}x^{k_i})$ for suitable $l$ and $a_i$ given in Table \ref{table_reembedding} by considering $I_{g,f} = \langle g, z_1-f_1, \dots  , z_m - f_m \rangle \in K[x^{\pm}, y^{\pm}, z_1^{\pm}, \dots, z_m^{\pm}]$ is faithful.
\end{theorem}
\begin{proof}
First, we separate the defining equation $g$ into its different components that identify the different (possibly shifted) building blocks using the coefficient conditions in Table \ref{table_standard_trop_coeff_cond}.
In Section \ref{sec_local_faithful}, we showed that for the first building block of the Berkovich skeleton of $\chi$, this is a suitable faithful re-embedding regardless of the type of the building block. By applying Proposition \ref{prop_shifting_building_blocks}, we see that this is true for all building blocks contained in the tropical curve. For two connected adjacent building blocks carrying genus, we proceed pairwise by elimination of variables:

We denote by $ I_{g,f_i, k_1}$ the re-embedding along $f_i$, by $I_{g,f_{i+1}, k_2}$ the re-embedding along $f_{i+1}$ and by $I_{g,f_i,f_{i+1}, k_1, k_2}$ the re-embedding along both $f_i$ and $f_{i+1}$. Since the two sets of variables of the re-embedding are disjoint, eliminating the variables $z_1, \dots, z_m$ we obtain $ I_{g,f_i,f_{i+1}, k_1, k_2} \cap K[x^{\pm}, y^{\pm}, w_1^{\pm}, \dots, w_{\tilde{m}}^{\pm}] = I_{g,f_{i+1}, k_2}. $
Analogously, we obtain that $I_{g,f_i,f_{i+1},k_1, k_2} \cap K[x^{\pm}, y^{\pm}, z_1^{\pm}, \dots, z_{m}^{\pm}] = I_{g,f_i, k_1}$. Since $ I_{g,f_i, k_1} \subset I_{g,f_i, f_{i+1}, k_1, k_2}$ and $ I_{g,f_{i+1}, k_2} \subset I_{g,f_i,f_{i+1}, k_1, k_2}$, all edges that appear in $ I_{g,f_i, k_1}$ and $ I_{g,f_{i+1}, k_2} $ can be identified with an edge of the same length in $I_{g,f_i,f_{i+1}, k_1, k_2}$. Thus, as $ I_{g,f_i, k_1}$  and $ I_{g,f_{i+1}, k_2} $  contain isometric copies of their respective building blocks, $I_{g,f_i,f_{i+1}, k_1, k_2}$ does as well.

By assumption, all multiplicities of edges in $I_{g,f_i, k_1}$ and $I_{g,f_{i+1}, k_2}$ corresponding to isometric copies of the building blocks of genus $k_1$ and $k_2$ are one. Thus, all multiplicities of edges in $I_{g,f_i,f_{i+1}, k_1, k_2}$ corresponding to isometric copies of the building blocks of genus $k_1$ and $k_2$ are one. Then, by Remark \ref{lem_bridges_faithful}, the tropicalization induced by  $I_{g,f_i,f_{i+1}, k_1, k_2}$ is faithful on both building blocks and their connector.
Thus, we obtain a faithful tropicalization of $\chi$.
\end{proof}

\begin{example}[Partial modification of two cycles]\label{ex_partial_mod_2_1_thetas}

Consider the curve $g(x,y) = y^2 - x(x-t^2)(x-t^4)(x-t^6)(x-t^8).$ All valuations of the monomials in $x$ are distinct. Using the coefficient conditions in Table \ref{table_standard_trop_coeff_cond}, we obtain that $g$ consists of two cycles, separated by a bridge.

First, we modify at the lower edge in direction $z$ with $f = y - \sqrt{t^{2}t^4t^6}x = y - t^6x$. Then, we modify at the upper edge using $\tilde{f} = y - \sqrt{t^2}x^2 = y-tx^2$, and obtain a faithful tropicalization. The resulting modification is depicted on the left in Example \ref{ex_1-thetas_reduced-dim}.
\end{example}

\subsection{Reducing dimensions}\label{sec_shrinking_dimensions}
We have a faithful re-embedding now, but we can do better, for curves whose Berkovich skeleton contains bridges: we can systematically cut down ambient dimension by combining multiple re-embedding polynomials into one.
\begin{theorem}\label{thm_mainthm_bridges}
Let $\chi$ be an irreducible hyperelliptic curve given by a defining equation $g$ as in Equation (\ref{eq_defining}) consisting of two $k$-thetas of genus $k_1$ and $k_2$ respectively, connected by a bridge. Let $I_{g,\hat{f}, k_1} = \langle g, z_1-\hat{f} \rangle \in K[x^{\pm}, y^{\pm}, z_1^{\pm}]$ and $I_{g,\tilde{f}, k_2} = \langle g, w_1-\tilde{f}_1 \rangle \in K[x^{\pm}, y^{\pm}, w_1^{\pm}]$ be the ideals determined in Section \ref{sec_local_faithful} for which the embedding is faithful, where $\hat{f} = y - \hat{h}(x)$ and $\tilde{f}= y-x^{k_1 + 1}\tilde{h}(x)$, obtained from Proposition \ref{prop_shifting_building_blocks}. Let $f = y -\hat{h}(x) - x^{k_1 + 1}\tilde{h}(x)$.

Then, the ideal $I_{g,f} = \langle g, z_1-f \rangle \in K[x^{\pm}, y^{\pm}, z_1^{\pm}]$ induces a faithful tropicalization.
\end{theorem}
\begin{proof}
By Theorem \ref{thm_mainthm_point_connectors_minimal_berkovich},  $I_{g,f, k_1, k_2} = \langle g, z_1-\hat{f}, w_1-\tilde{f} \rangle \in K[x^{\pm}, y^{\pm}, z_1^{\pm}, w_1^{\pm}]$ induces a faithful tropicalization on the whole Berkovich skeleton.

Let $F=\trop(f)$. As $\hat{F}= \trop(\hat{f})$ and $\tilde{F}= \trop(\tilde{f})$ pass through the respective multiplicity two edges of the thetas in the standard $xy$-tropicalization of $g$ by Section \ref{sec_reembeddings_and_blocks}, so does $F$ by its construction.

Let $C$ be the additional cell in $F$ that is not contained in any cell of $\hat{F}$ or $\tilde{F}$. (For an example, see Example \ref{ex_1-thetas_reduced-dim})

Outside of $C$, the modification agrees with  $I_{g,\hat{f}, k_1}$ before the bridge and $I_{g,\tilde{f}, k_2}$ after the bridge respectively. They are thus faithful on the two thetas as they lie on distinct cells. Hence, by Remark \ref{lem_bridges_faithful}, the modification is faithful on the minimal Berkovich skeleton.
By construction, there is no edge or leg on $C$. Further, all legs and edges of the projected curve around $C$ have multiplicity one by Remark \ref{lem_bridges_faithful}. By Lemma \ref{lem_proj_lemma_genus_n}, this extends to the re-embedding.  Thus, the modification is faithful on $C$ as well.
\end{proof}
\begin{remark}\label{thm_mainthm_topdim_cones}
If the Berkovich skeleton of $\chi$ consists only of thetas and bridges (i.e. $\chi_{trop}$ is in a \emph{top-dimensional cone of the moduli space $M_{3,n}^{trop}$}), the tropicalization of $\chi$ can be faithfully modified in dimension three by re-embedding Trop($\chi$) into a tropical hypersurface obtained as the tropicalization of $V(z- (y - a_1x - a_2x^2 - \dots + a_7x^7) \in K[x,y]$.
\end{remark}

\begin{example}[Global tropical modification of two cycles]\label{ex_1-thetas_reduced-dim}
Consider the curve $g(x,y) = y^2 - x(x-t^2)(x-t^4)(x-t^6)(x-t^8)$ from Example \ref{ex_partial_mod_2_1_thetas}. By using reducing dimension using Theorem \ref{thm_mainthm_topdim_cones}, we can further simplify Example \ref{ex_partial_mod_2_1_thetas} while preserving faithfulness. By modifying $\R^2$ at $f = y- t^6x - tx^2$, we obtain a re-embedding that is faithful on the tropical curve, see Figure \ref{fig: modifications}.
\end{example}

\begin{figure}
    \centering
\begin{tabular}{c c}
\includegraphics{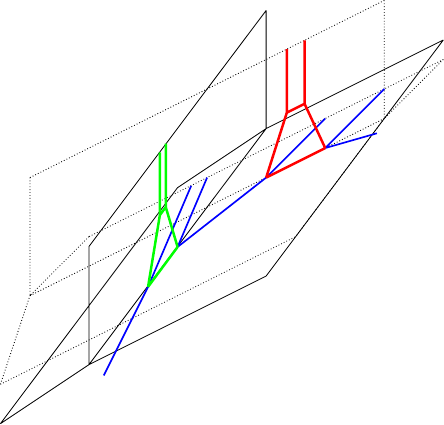}
&\includegraphics{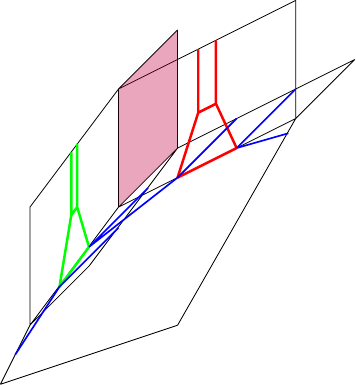}
\end{tabular}
    \caption{
The picture on the left shows the modification of Example \ref{ex_partial_mod_2_1_thetas}, and the picture on the right shows the combined modification. The purple cell on the right is $C$ referred to in the proof of Theorem \ref{thm_mainthm_bridges}.}
    \label{fig: modifications}
\end{figure}

\bibliographystyle{abbrv}
\bibliography{references}

\end{document}